\documentclass{amsart}

\usepackage{graphicx}
\usepackage{amssymb}
\usepackage{amsmath}
\usepackage[all]{xy}
\usepackage{txfonts}
\theoremstyle{plain}
\newtheorem{theorem}{Theorem}[section]
\newtheorem*{theorem*}{Theorem}

\newtheorem*{maintheorem*}{Main Theorem}
\newtheorem{proposition}[theorem]{Proposition}%[section]
\newtheorem{corollary}[theorem]{Corollary}%[section]
\newtheorem{lemma}[theorem]{Lemma}%[section]

\newtheorem*{conjecture*}{Conjecture}%[section]

\theoremstyle{definition}
\newtheorem{definition}[theorem]{Definition}
\newtheorem*{definition*}{Definition}%[section]
\newtheorem*{example*}{Example}
\newtheorem*{notation*}{Notation}
\newtheorem*{notation-conv*}{Notation and convention}
\newtheorem*{convention*}{Convention}
\theoremstyle{remark}
\newtheorem{remark}[theorem]{Remark}%[sectio]

\newcommand{\Z}{{\mathbb Z}}
\newcommand{\C}{{\mathbb C}}

\newcommand{\I}{\mathbf{1}}

\newcommand{\SU}{{\mathrm{SU}(2)}}
\newcommand{\SO}{{\mathrm{SO}(3)}}
\newcommand{\SL}{{\mathrm{SL}_2(\C)}}
\newcommand{\PSL}{{\mathrm{PSL}_2(\C)}}

\newcommand{\trace}{{\rm tr}\,}
\newcommand{\trans}[1]{{}^t\!#1}
\newcommand{\tr}{\mathop{\mathrm{tr}}\nolimits}
\newcommand{\im}{\mathop{\mathrm{Im}}\nolimits}
\newcommand{\bm}[1]{\mbox{\boldmath{$#1$}}}
\newcommand{\hiota}{\widehat{\iota}}
\newcommand{\hPhi}{\widehat{\Phi}}
\newcommand{\wtilde}[1]{\widetilde{#1}}
\newcommand{\co}{\colon\thinspace}

\newcommand{\knotexterior}{E_K}
\newcommand{\branchedcoverZHS}[1]{B_{#1}}
\newcommand{\homPiICnEk}{p_*}
\newcommand{\nclosure}[1]{\langle\!\langle #1 \rangle\!\rangle}
\newcommand{\PhiII}{\Phi}
\newcommand{\hatPhiII}{\widehat{\Phi}}
\newcommand{\interior}{{\rm int}\,}

\newcommand{\ctransf}{\sigma}
\newcommand{\pullbacki}{i^*}
\newcommand{\pullbackj}{j^*}
\newcommand{\pullbackp}{p^*}
\newcommand{\pullbackq}{q^*}
\newcommand{\isomInVanKampen}{h}

\begin{document}

\title[On the meridonal trace--free slice of the character variety]
{On the geometry of the slice of trace--free $\SL$-characters of a knot group}

\thanks{
  The first author had been partially
  supported by JSPS Research Fellowships for Young Scientists and 
  the Grant-in-Aid for Young Scientists (Start-up). The second author had been partially
  supported by the 21st century COE program at Graduate School of
  Mathematical Sciences, University of Tokyo.
}

\author{Fumikazu Nagasato \and Yoshikazu Yamaguchi}

\address[F. Nagasato]{
              Department of Mathematics, Meijo University, 
              Shiogamaguchi Tempaku, Nagoya 468-8502, Japan
            }
 \email{fukky@meijo-u.ac.jp}
\address[Y. Yamaguchi]{
              Department of Mathematics, Tokyo Institute of Technology
              2-12-1 Ookayama, Meguro-ku Tokyo, 152-8551, Japan
}
\email{shouji@math.titech.ac.jp} 

\date{}

%%%%%%%%%%%%%%%%%%%%%%%%%%%%%%%%%%%%%%%%%% 
% Subject classification 
%%%%%%%%%%%%%%%%%%%%%%%%%%%%%%%%%%%%%%%%%% 

%%%%%%%%%%%%%%%%%%%%%%%%%%%%%%%%%%%%%%%%%% 
% abstract
%%%%%%%%%%%%%%%%%%%%%%%%%%%%%%%%%%%%%%%%%% 
\begin{abstract}

  Let $K$ be a knot in an integral homology $3$--sphere $\Sigma$ with
  exterior $E_K$, and let $B_2$ denote the $2$--fold branched cover of
  $\Sigma$ branched along $K$. We construct a map $\Phi$ from the
  slice of characters with trace free along meridians in the
  $\SL$-character variety of the knot exterior $E_K$ to the
  $\SL$-character variety of $2$--fold branched cover $B_2$. When this
  map is surjective, it describes the slice as the $2$--fold branched
  cover over the $\SL$-character variety of $B_2$ with branched locus
  given by the abelian characters, whose preimage is precisely the set
  of metabelian characters. We show that each of metabelian character
  can be represented as the character of a binary dihedral
  representation of $\pi_1(E_K)$. The map $\Phi$ is shown to be
  surjective for all 2-bridge knots and all pretzel knots of type $(p,
  q, r)$. An extension of this framework to $n$--fold branched covers
  is also described.
\end{abstract}

%%%%%%%%%%%%%%%%%%%%%%%%%%%%%%%%%%%%%%%%%%%%%%%%%%%%%%%%%%%%%%%%%%%% 
% body of paper
%%%%%%%%%%%%%%%%%%%%%%%%%%%%%%%%%%%%%%%%%%%%%%%%%%%%%%%%%%%%%%%%%%%% 

\maketitle

\section{Introduction}
The purpose of this paper is to explore the relationship between
$\SL$-representations for a homology knot exterior and those for a
finite cyclic cover over an integral homology sphere $\Sigma$ with
branch set the knot.  We define a correspondence from
$\SL$-representations of the fundamental groups of knot exteriors with
trace of the meridian fixed to those of finite cyclic covers branched
along the knots. Then we describe features of this correspondence for knots
in $S^3$ and $2$--fold branched covers in terms of character
varieties.

In the interaction between knot theory and $3$-dimensional
topology, special values of polynomial invariants of knots often give
good expressions for topological invariants of closed $3$-manifolds.
Fox's formula is a famous bridge between knot theory and
three dimensional topology.  This formula says that the product of
values of the Alexander polynomial of a knot at roots of unity gives
the order of the first homology group of finite cyclic cover branched
along the knot.  Behind such a phenomenon, the representation space
and the character variety of the fundamental group of a manifold has
played an important role. Our observation was motivated to understand
the correspondence between representations and characters of the
fundamental groups of knot exteriors and finite cyclic covers with 
branch set the knot behind Fox's formula.

The key to such a correspondence is choosing a subset of
representation spaces and character varieties for knot exteriors. 
The $G$-character variety of a manifold $M$ can be considered as the set
of characters of $G$-representations of the fundamental group
$\pi_1(M)$ with a structure of affine variety (for a
precise definition for $G=\SL$, refer to Subsection
$\ref{subsection:CV}$ or~\cite{CS:1983}).  
Here we recall some beautiful results using the $\SU$-character varieties 
shown by A. Casson and X.-S. Lin. 

Casson introduced an invariant for an integral homology $3$--sphere $\Sigma$, 
so-called the Casson invariant, originally by using the intersection 
between $\SU$-character varieties associated to a Heegaard splitting 
of $\Sigma$. More precisely, the character varieties of two handlebodies 
associated to a Heegaard splitting of $\Sigma$ are embedded 
in that of the common boundary surface of the handlebodies. 
Then the Casson invariant is defined as a half of the integer obtained 
by counting the algebraic intersection of the embedded 
character varieties of the handlebodies.
(see~\cite{Akbulut-MacCartht,saveliev99:_lectur_topol_of_manif} for
expositions). X.-S. Lin~\cite{Lin:CassonLin} defined an invariant for
a knot $K$ in $3$--sphere $S^3$, now known as the Casson--Lin
invariant, by applying similar idea to the $\SU$-character variety of
the fundamental group of the knot exterior $\knotexterior$ with 
{\it trace free along meridians}, whose representations send meridians to
trace zero matrices. Lin showed that the Casson--Lin invariant is a
half of the signature of the knot. 

Afterward, C. Herald generalized
the Casson--Lin invariant by using gauge theoretic methods
in~\cite{herald97:_flat_connec_alexan_invar_and_casson} and
M. Heusener and J. Kroll also have studied the same issue by
topological methods in~\cite{HeusenerKroll98} to consider other trace condition
of meridians indexed by $t \in (-2, 2)$. The knot invariant by
Herald's and Heusener--Kroll's generalization corresponds to the
equivariant signature of the knot.
On the other hand,
O. Collin and N. Saveliev~\cite{CollinSaveliev01}
have studied $\SU$-character varieties of knot exteriors 
with trace conditions of meridians from a viewpoint of gauge theory with cyclic 
group actions. They considered finite cyclic branched covers 
over integral homology $3$--spheres with branch set a knot
and define a topological invariant, called the
equivariant Casson invariant, for integral homology $3$--spheres with
finite cyclic group actions. Then it turned out that the
equivariant Casson invariant can be identified with the equivariant knot signatures.
Collin--Saveliev's approach implies that 
the $\SU$-character variety of a knot exterior with
trace conditions of meridians can be related to that of a finite cyclic branched cover.

In this perspective, for a knot $K$ in $\Sigma$, we construct a map $\Phi$ from
$\SL$-representations of $\pi_1(\knotexterior)$ with trace free along meridians
to those of the fundamental group of the $2$--fold branched cover 
$\branchedcoverZHS{2}$ over $\Sigma$ branched along $K$. 
The domain of this map $\Phi$ contains
{\it metabelian representations} of the knot group
$\pi_1(\knotexterior)$ and $\Phi$ sends the metabelian
representations to abelian representations of
$\pi_1(\branchedcoverZHS{2})$ (refer to
Proposition~\ref{prop:correspond_rep}). 
This map $\Phi$ naturally induces a map $\hPhi$ between characters of $\pi_1(\knotexterior)$ and $\pi_1(\branchedcoverZHS{2})$.

The domain of $\hPhi$ is the subset with trace free along meridians, 
denoted by $S_0(\knotexterior)$.
Such a subset of $\SL$-characters with the trace of meridians fixed is
called a {\it slice} since it is a level set of the
function given by the evaluation of trace along meridians at the fixed values. 
In Theorem~\ref{thm:geom_S0}, Propositions~\ref{prop:one_to_one} and 
\ref{prop:two_to_one}, we will see some interesting properties of $\hPhi$ and 
$S_0(\knotexterior)$ for 3--sphere $S^3$. 
In our description of $\hPhi$, the characters of metabelian representations 
have the significant feature and role. 
They form the fixed point set of an involution on $S_0(\knotexterior)$ 
(refer to Proposition~\ref{prop:fixed_point_set})
and they are sent to the characters of abelian representation of 
$\pi_1(\branchedcoverZHS{2})$ by $\hPhi$ one-to-one.
We will also see that $\hPhi$ is surjective for all $2$-bridge knots 
and all pretzel knots of type $(p, q, r)$ and then
$S_0(\knotexterior)$ is the $2$--fold branched cover over the character variety of $\branchedcoverZHS{2}$ via $\hPhi$.

The above framework for a $2$--fold branched cover 
$\branchedcoverZHS{2}$ are naturally generalized to the case 
of an $n$--fold cyclic branched cover $\branchedcoverZHS{n}$. 
Namely, we can also define a map from the subset 
$R_{2 \cos (\pi k/n)}(\knotexterior)$ of $\SL$-representations of $\pi_1(\knotexterior)$ 
with trace $2 \cos (\pi k/n)$ along meridians 
to $\SL$-representations $R(\branchedcoverZHS{n})$ of $\pi_1(\branchedcoverZHS{n})$.
We can descend it to a map from 
the slice $S_{2 \cos (\pi k/n)}(\knotexterior)$ of characters,
associated with $R_{2 \cos (\pi k/n)}(\knotexterior)$,
to characters of $\pi_1(\branchedcoverZHS{n})$. 
This extension of framework is another main result in this article.

\paragraph{Organization}
In Section~\ref{section:preliminaries}, we review some notions about
representations and the characters. Then we define notations used
throughout this article. Section~\ref{section:Lin_presentation} gives
a review of Lin presentation of a knot group and study its properties.
In Section~\ref{section:char_metabelian}, we discuss the characters of
metabelian representations and see  
important properties of {\it binary dihedral} representations in metabelian ones.
In Section~$\ref{section:map_rep_spaces}$, we derive how to make an 
$\SL$-representation of $\pi_1(\branchedcoverZHS{2})$ from that of 
$\pi_1(\knotexterior)$. In Section~\ref{section:slice_S0}, we focus on the
structure of the slice $S_0(\knotexterior)$ with trace free along
meridians for 3--sphere $S^3$ via the correspondence to the characters of $\pi_1(\branchedcoverZHS{2})$. 
In the final section, we show several applications including 
the surjectivity of the map $\hatPhiII$ 
for two--bridge knots and pretzel knots of type $(p,q,r)$.

\section{Preliminaries}\label{section:preliminaries}
\subsection{General notations}\label{subsec:general_notations}
We use the symbol $\Sigma$ to denote an integral homology
$3$--sphere and $K$ to denote a knot in $\Sigma$.  Then
$\knotexterior$ stands for the knot exterior, i.e., it is obtained
by removing an open tubular neighbourhood of $K$ from $\Sigma$.  Let
us denote by $\branchedcoverZHS{n}$ the $n$--fold cyclic branched
cover of $\Sigma$ along the branched set $K$ and by $C_n$ the
$n$--fold cyclic cover of $\knotexterior$.  The following diagram
expresses the relationship among these spaces,
$$
\xymatrix{
C_n \ar[r] \ar[d] & \branchedcoverZHS{n} \ar[d] \\
\knotexterior \ar[r] & \Sigma.
}
$$
The right arrows are inclusions and the down arrows are projections.
When we consider several fundamental groups $\pi_1(\Sigma)$,
$\pi_1(\knotexterior)$, $\pi_1(C_n)$ and
$\pi_1(\branchedcoverZHS{n})$, we assume that the base points of
$\Sigma$ and $\knotexterior$ are given by the same point $b$ which
lives in the boundary torus of $\knotexterior$ and those of $C_n$
and $\branchedcoverZHS{n}$ are given by a lift $\hat{b}$ of $b$.  We
will abbreviate $\pi_1(\Sigma, b)$ to $\pi_1(\Sigma)$ and so on.  We
denote by $\mu$ a meridian of a knot $K$, which is the element in
$\pi_1(\knotexterior)$ represented by a simple closed curve in the
boundary torus.

\subsection{Review on the $\SL$-character variety}
\label{subsection:CV}
We briefly review the $\SL$-representation space and the character
variety.  We consider $\SL$-representations of the fundamental group
of a compact manifold $M$, i.e., homomorphisms from $\pi_1(M)$ into
$\SL$.  We let $$R(M) = \mathrm{Hom}(\pi_1(M); \SL)$$ denote the space
of $\SL$-representations for $M$.  We can see that $R(M)$ is an affine
algebraic set from a presentation of $\pi_1(M)$, by viewing $R(M)$ as
solutions to polynomial equations in a product of copies of $\SL$ for
each generator in $\pi_1(M)$.

A representation $\rho \co \pi_1(M) \to \SL$ is called \emph{abelian} if
the image $\rho(\pi_1(M))$ is an abelian subgroup in $\SL$.  We say
that $\rho$ is {\it metabelian} if the commutator subgroup $[\pi_1(M),
\pi_1(M)]$ is sent to an abelian subgroup in $\SL$ by $\rho$.  Of
course, all abelian representations are metabelian.
It is well--known that there exist only the following two maximal
abelian subgroups, up to conjugation, in $\SL$:
\[
  {\rm Hyp} := 
    \left\{ \left.  
      \begin{pmatrix}
        \lambda & 0 \\
        0 & \lambda^{-1}
      \end{pmatrix}
      \,\right|\, 
      \lambda \in \C^{*} :=\C \setminus \{0\}
    \right\},\quad
  {\rm Para} := 
    \left\{ \left.  
      \pm 
      \begin{pmatrix}
        1 & \omega \\
        0 & 1
      \end{pmatrix}
      \,\right|\, 
      \omega \in \C 
    \right\}.
\]
Here ${\rm Hyp}$ means hyperbolic elements in $\SL$ and ${\rm Para}$
means parabolic elements in $\SL$.  
Hence if $\rho$ is abelian
(resp. metabelian), the image (resp. the image of the commutator subgroup 
$[\pi_1(M), \pi_1(M)]$) can be contained in ${\rm Hyp}$ or 
${\rm Para}$ by taking conjugation.

An $\SL$-representation $\rho$ is called \emph{reducible} if there
exists an invariant line $L \subset \C^2$ such that $\rho(g)(L)
\subset L$, for all $g \in \pi_1(M)$.  Namely, there exists an element
$A$ of $\SL$ such that $A \rho(g) A^{-1}$ is an upper triangular
matrix for any $g\in \pi_1(M)$.  Of course, any abelian representation
is reducible (while the converse is false in general).  Moreover by
\cite[Lemma 1.2.1]{CS:1983}, all reducible representations send the
commutator subgroup $[\pi_1(M), \pi_1(M)]$ into the maximal abelian
subgroup ${\rm Para}$, up to conjugate. Hence reducible
representations are metabelian.  A representation is called 
{\it irreducible} if it is not reducible.

The group $\SL$ acts on the representation space $R(M)$ by
conjugation, however the naive quotient $R(M)/\SL$ is not Hausdorff in
general.
To construct an appropriate quotient, we take the affine
algebraic set determined by the condition that its coordinate ring is
isomorphic to the ring of invariant regular functions
$\C[R(M)]^{\PSL}$.  By M. Culler and P. Shalen~\cite{CS:1983}, we can
endow the set of characters of representations with such affine
algebraic structure. Here the character associated to $\rho \in R(M)$
is the map $\chi_\rho \co \pi_1(M) \to \C$, defined by $\chi_\rho(g) =
\tr \rho(g)$ where $\tr$ denotes the trace of matrices.
In this sense, the set of characters of $R(M)$ is
called the character variety of $M$ and denoted by $X(M)$.  

Let ${R}^\mathrm{irr}(M)$ denote the subset of irreducible
representations of $\pi_1(M)$, and $X^\mathrm{irr}(M)$ denote its
image under the map $t \co {R}(M) \to X(M)$, $t(\rho)=\chi_\rho$.  Note
that two irreducible representations of $\pi_1(M)$ in $\SL$ with the
same character are conjugate by an element of $\SL$,
see~\cite[Proposition 1.5.2]{CS:1983}.
Similarly, we write $X^{\mathrm{red}}(M)$ for the image of the set
$R^{\mathrm{red}}(M)$ of reducible representations under $t$.  We also
use $R^{\mathrm{ab}}(M)$ (resp. $R^{\mathrm{meta}}(M)$) for the set of
abelian (resp. metabelian) representations and $X^{\mathrm{ab}}(M)$
(resp. $X^{\mathrm{meta}}(M)$) for the image by $t$.  
We refer the character of an irreducible representation to 
an {\it irreducible character} and 
similarly for reducible, abelian, metabelian character and so on. 

\section{Free Seifert surfaces and Lin presentations}
\label{section:Lin_presentation}
X.-S. Lin has introduced a new method to represent generators and
relations of a knot group by using a free Seifert surface and the
induced Heegaard splitting of $S^3$.  We call this method {\it Lin
  presentation}.  
His approach also induces presentations of the fundamental groups of
covering spaces and allows us to investigate
character varieties of cyclic covers over a knot exterior.
We will use these presentations as a main tool for our study in the rest
of this article.  In subsection~\ref{subsection:def_Lin_presentation},
we review free Seifert surfaces and Lin presentations. In
subsections~\ref{subsection:pres_piI_Cn} and
\ref{subsection:pres_piI_Bn}, we show how a Lin presentation of
$\pi_1(\knotexterior)$ determines presentations of the fundamental
groups $\pi_1(C_n)$ and $\pi_1(\branchedcoverZHS{n})$ for the
$n$--fold cyclic and branched covers.

\subsection{Lin presentation for an integral homology sphere}
\label{subsection:def_Lin_presentation}
Lin presentation was given for a knot in $S^3$, however, his
construction can be extended for a knot in an integral homology
$3$--sphere $\Sigma$.  To explain this extension, we review his construction.

We start with the definition of a free Seifert surface and existence
of such a Seifert surface in $\Sigma$.
\begin{definition}
  A Seifert surface $S$ of a knot $K$ is called free 
  if $\Sigma = N(S) \cup \overline{\Sigma \setminus N(S)}$ is a Heegaard splitting 
  of $\Sigma$.  Here $N(S)$ is a closed
  tubular neighborhood of $S$ in $\Sigma$.
\end{definition}
\begin{lemma}\label{lemma:preferred_Heegaard_splitting}
  Let $K$ be a knot in $\Sigma$.  Then there exists an embedded
  surface $S$ in $\Sigma$ such that $S \times [-1, 1] \cup
  \overline{\Sigma \setminus S \times [-1, 1]}$ is a Heegaard
  splitting of $\Sigma$ satisfying $K = \partial S$.
\end{lemma}
For more information about the existence of a free Seifert surface and
the induced Heegaard splitting of $\Sigma$, for example,
see~\cite[Lemma~17.2]{saveliev99:_lectur_topol_of_manif}.

A presentation of $\pi_1(\knotexterior)$ is obtained from the
Heegaard splitting associated to a free Seifert surface.

\begin{lemma}[Lemma 2.1 in \cite{Lin01}]\label{lemma:Lin_presentation}
  Assume that $S$ is a free Seifert surface for $K$ in $\Sigma$.  Let
  $\Sigma = H_1 \cup H_2$ be the Heegaard splitting associated to $S$. 
  For a basis $x_1, \ldots, x_{2g}$ of the free group $\pi_1(H_2)$, 
  we have the following presentation of
  $\pi_1(\knotexterior)$:
  \begin{equation}\label{eqn:Lin_equation}
    \pi_1(\knotexterior)=
    \langle 
    x_1, \ldots, x_{2g}, \mu \,|\, \mu\alpha_i\mu^{-1} = \beta_i, i=1, \ldots, 2g
    \rangle,
  \end{equation}
  where $\alpha_i$, $\beta_i$ are words in $x_1, \ldots, x_{2g}$,
  determined by the spine of $S$, and $g$ is the genus of $S$.
\end{lemma}
We call the presentation~$(\ref{eqn:Lin_equation})$ a Lin presentation of $\pi_1(\knotexterior)$.
\begin{remark}
  Note that every $x_i$ is contained in the commutator subgroup of
  $\pi_1(\knotexterior)$.
\end{remark}

\begin{proof}[Proof of Lemma~\ref{lemma:Lin_presentation}]
  This can be shown by van Kampen's theorem.  We begin by a
  decomposition of $\knotexterior$ associated with a Heegaard
  splitting of $\Sigma$ by removing a solid torus inside single
  handlebody, which is isotopic to an open tubular neighbourhood of
  $K$.  As in Lemma~\ref{lemma:preferred_Heegaard_splitting}, the
  integral homology sphere $\Sigma$ is decomposed into two
  handlebodies $H_1$ and $H_2$ where $H_1 = S \times [-1, 1]$.  By
  pushing $K$ into $H_1$ slightly, we have a simple closed curve $K'$
  on $S \times \{0\}$, which is parallel to $\partial S = K$.  We
  identify $E_K$ with $E_{K'}$ and apply van Kampen's theorem to the
  decomposition $E_K = (H_1 \setminus \interior N(K)) \cup H_2$.  Then
  van Kampen's theorem says that the knot group
  \[ \pi_1(\knotexterior) \simeq \pi_1(H_1\setminus \interior N(K))
  \mathop{\ast}_{\pi_1(\partial H_1)}\pi_1(H_2) .\] and by homotopy
  equivalence and van Kampen's theorem again, we have
  \begin{align*}
    \pi_1(H_1\setminus \interior N(K)) 
    &\simeq \pi_1(\partial N(K)) \mathop{\ast}_{\pi_1(\hbox{\scriptsize longitude})}\pi_1(S) \\
    &= \langle \mu, a_1, \ldots, a_{2g} \,|\, [ \mu, \prod_{i=1}^{g}[a_{2i-1},a_{2i}] ] = 1 \rangle
  \end{align*}
  where $\mu$ is a meridian and $a_1, \ldots,
  a_{2g}$ are circles on $S$ and give a spine of $S$, which is a
  deformation retract of $S$ and $g$ is the genus of $S$. 

  Hence $\pi_1(\knotexterior)$ is generated by $x_1, \ldots,
  x_{2g}$, $a_1, \ldots, a_{2g}$ and $\mu$. But since every $a_i$ is
  written by a word in $x_1, \ldots, x_{2g}$ in $\pi_1(E_K)$ by the
  homeomorphism between $\partial H_1$ and $\partial H_2$, we can
  reduce the generators to $x_1, \ldots, x_{2g}, \mu$.  The relations
  of $\pi_1(\knotexterior)$ are given by disks in $H_1 \setminus
  \interior N(K)$, whose boundaries belong to $\partial
  (H_1 \setminus \interior N(K)) \cup S$.  We can assume that such a disk $D^2$ in $H_1
  \setminus \interior N(K)$ intersects with $S$
  transversally.  Then the homotopy class of $\partial D^2$ is given
  by a word in $\mu a_i^+ \mu^{-1} (a_i^{-})^{-1}$ ($i=1, \ldots, 2g$)
  where $a_i^{\pm}$ is $a_i \times \{\pm 1\}$ respectively.  The relations of
  $\pi_1(E_K)$ are generated by $\mu a_i^+ \mu^{-1} (a_i^{-})^{-1}$
  ($i=1, \ldots, 2g$).  So, by rewriting $a_i^{\pm}$ as words
  $\alpha_i$ and $\beta_i$ in $x_1, \ldots, x_{2g}$, we obtain the
  presentation in Lemma~\ref{lemma:Lin_presentation}. 
\end{proof}

\begin{figure}[!ht]
\begin{center}
  \includegraphics[scale=.3]{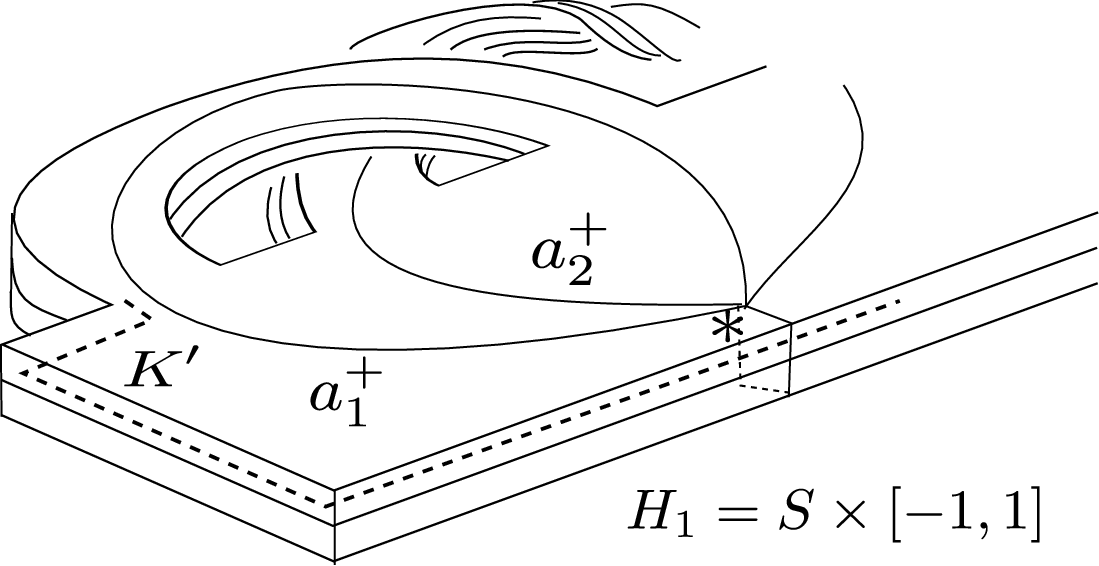}
\end{center}
\caption{Seifert surface}\label{fig:Seifert}
\end{figure}

Let us denote by $v_{i,j}$ the exponent sum of $x_j$ in $\alpha_i$ and
$u_{i,j}$ the exponent sum of $x_i$ in $\beta_i$.  We set two $(2g
\times 2g)$-matrices $V := (v_{i,j})$ and $U := (u_{i,j})$ with integer
entries.

For a knot in $S^3$, Lin~\cite{Lin01} has introduced special free Seifert surface,
called {\it regular}.  A free Seifert surface $S$
is called regular if it has a spine where embedding in $S^3$ induced
by $S$ is isotopic to the standard embedding.  By setting appropriate
orientations on $a_i$ and $x_j$, we obtain the following standard
properties of $V$ and $U$:
\begin{enumerate}
\item $U$ is so-called the Seifert matrix $Q=(Q_{i,j})$, $Q_{i, j} =
  lk(a_i, a_j^{+})$ for $S$;
\item $V = \trans{U}$, where $\trans{U}$ is the transpose matrix of
  $U$.
\end{enumerate}
Note that Lin used the matrix $(lk(a_i^+, a_j))_{i, j}$ as
the Seifert matrix, which is different from ours $(lk(a_i, a_j^{+}))_{i,j}$.  
This gives rise to differences between
the convention in~\cite{Lin01} of matrices $U$ and $V$ and ours.

Unfortunately, regular Seifert surfaces are defined only for knots in
$S^3$.  However, $V$ and $U$ are expressed by using the Seifert matrix
for a general free Seifert surface of a knot in an integral homology
sphere $\Sigma$.  To describe $V$ and $U$ for a free Seifert surface
in $\Sigma$, we need one more $(2g \times 2g)$-matrix $T$ which
represents a boundary operator in the Morse complex $\oplus_{i=0}^3
C^{\hbox{\scriptsize Morse}}_i(\Sigma;\Z)$, where
$C^{\hbox{\scriptsize Morse}}_i(\Sigma;\Z)$ is generated by critical
points with index $i$ of a Morse function associated to the Heegaard
splitting as follows.

Let $S$ be a free Seifert surface for a knot $K$ in $\Sigma$.  Set
$\Sigma = H_1 \cup H_2$ as a Heegaard splitting associated to $S$.
The set of cores of $1$-handle in $H_1$ are given by $\{a_1, \ldots, a_{2g}\}$ 
in the spine $W_{2g}$.  We denote by $D_i$ a meridian disk for $a_i$. 
The cores of $2$-handles represent the  generators
$\{x_1, \ldots, x_{2g}\}$ in $\pi_1(H_2)$. We denote by $D'_i$ a
meridian disk for $x_i$.

It is known that there exists a Morse function $f \co Y \to [0, 3]$ {\it compatible}
with the Heegaard splitting (see for
instance~\cite{ozsvath04:_holom_disks_and_topol_invar}),  i.e.,
$f$ is a self--indexing Morse
function on $\Sigma$ with one minimum and one maximum and $f$ induces
the Heegaard decomposition with surface 
$S = f^{-1}(3/2)$, $H_1 = f^{-1}([0, 3/2])$, $H_2 = f^{-1}([3/2, 3])$. 
The attaching circles $\partial D_i$ and $\partial D'_i$ are the 
intersections of $S$ with the ascending and descending gradient flows from 
the index one and two critical points respectively.

The intersection point between the meridian disk $D_i$ and the core
$a_i$ gives a critical point with index one.  We denote by $q_i$ this
critical point.  The intersection $D'_i$ with $x_i$ gives a critical
point with index two. We denote it by $p_i$.  Each $2$-handle is
attached to $H_1$ along the image $\partial D'_i$ in $\partial H_1 = S$.
We set an integer $t_{k, j}$ as the intersection number $\partial D_j$
with $\partial D'_k$ in $\partial H_1$.  The matrix $(t_{k, j})$ gives
the boundary operator from 
$C^{\hbox{\scriptsize Morse}}_2(\Sigma;\Z)$ to 
$C^{\hbox{\scriptsize Morse}}_1(\Sigma ;\Z)$.  
The correspondence between two bases $(p_1, \ldots, p_{2g})$ and 
$(q_1, \ldots, q_{2g})$ is expressed as
$$
\partial (p_1, \ldots, p_{2g}) = (q_1, \ldots, q_{2g})\, T, 
$$
where $T$ is the $(2g \times 2g)$-matrix $(t_{k, j})$.

\begin{figure}[!ht]
\begin{center}
  \includegraphics[scale=.5]{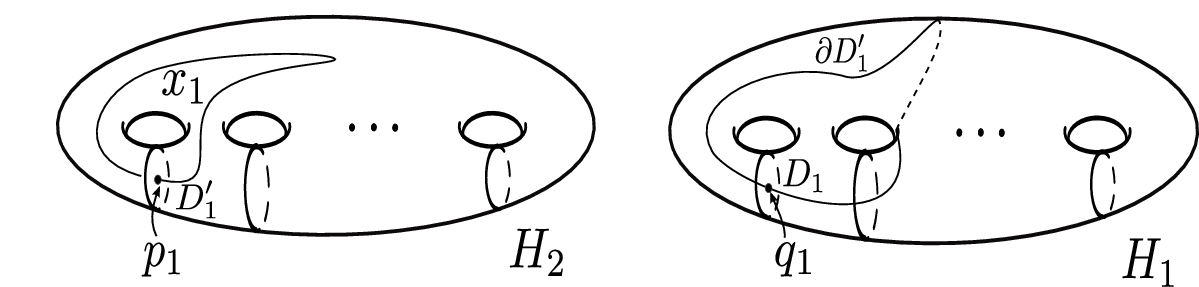}
\end{center}
\caption{Heegaard splitting}\label{fig:Heegaard_splitting}
\end{figure}

\begin{remark}
  Since $\Sigma$ is an integral homology $3$--sphere, the
  representation matrix $T$ is invertible.  In particular the
  determinant $\det(T)$ is equal to $\pm 1$.
\end{remark}

\begin{proposition}\label{prop:general_U_V}
  We suppose that $S$ is a free Seifert surface of a knot $K$.  Let
  $V$ and $U$ be the matrices defined by the relations $\alpha_i$ and
  $\beta_i$ in Lin presentation $(\ref{eqn:Lin_equation})$ of the knot
  group.  Then the matrices $V$ and $U$ are expressed as
  \[V = \trans{Q} T, \quad U = Q T,\] 
  where $Q$ is the Seifert matrix
  for $S$ and $T$ is the matrix corresponding to the boundary operator
  from $C^{\hbox{\scriptsize Morse}}_2(\Sigma;\Z)$ to
  $C^{\hbox{\scriptsize Morse}}_1(\Sigma;\Z)$.
\end{proposition}
\begin{proof}
  From the definition, the integer $v_{i, j}$ is given by the
  intersection number of $a_i^+$ with $\partial D'_j$ in $\partial H_1$.
  Since $\partial D'_j$ is homologue to $\sum_{k=1}^{2g} t_{k, j}
  a_k$, by using linking number, the entry $v_{i, j}$ in $V$ is
  expressed as
  \begin{equation}\label{eqn:entry_in_V}
    v_{i, j} = lk(a_i^+,\, \sum_{k=1}^{2g} t_{k, j}a_k),
  \end{equation}
  where $t_{k, j}$ is the entries in the matrix $T$.  The right hand
  side of $(\ref{eqn:entry_in_V})$ turns into
  \begin{align*} 
    lk(a_i^+,\, \sum_{k=1}^{2g} t_{k, j} a_k) 
    &= \sum_{k=1}^{2g} t_{k, j} lk (a_i^+,\, a_k) \\
    &= \sum_{k=1}^{2g} t_{k, j} Q_{k, i}.
  \end{align*}
  This means that $V=\trans{Q}T$. Similarly we can prove that
  $U=QT$.
\end{proof}

\begin{remark}\label{remark:Alexander_poly}
  The Alexander polynomial $\Delta_K(t)$ is given by $\det (t\*U-V)$.
\end{remark}
\begin{remark}\label{remark:matrix_T_identity}
  If $K$ is a knot in $S^3$ and $S$ is regular, then the matrix $T$ is
  the identity matrix.
\end{remark}

\subsection{The induced presentation for the fundamental group of the $n$--fold 
cyclic cover}\label{subsection:pres_piI_Cn}
The main results in this subsection are Lemmas~\ref{lemma:presentation_for_C_n} 
and~\ref{lemma:presentation_homology_C_n}. 
Lemma~\ref{lemma:presentation_for_C_n} describes a presentation 
for $\pi_1(C_n)$ induced by a choice of 
Lin presentationfor $\pi_1(\knotexterior)$ and Lemma 4 establishes a useful relation 
in the homology $H_1(C_n;\Z)$. 
%We describe a presentation of $\pi_1(C_n)$ induced by a Lin presentation
%of $\pi_1(\knotexterior)$ and establish a useful relation in the
%homology $H_1(C_n;\Z)$.

The $n$--fold cyclic cover $C_n$ over $\knotexterior$ has the
covering transformations, denoted by $\ctransf$.  The covering
transformations form a cyclic group with order $n$.  This action
induces the automorphism $\tau$ of $\pi_1(C_n)$,
\begin{align*}
  \tau \co \pi_1(C_n, \hat{b}) &\to \pi_1(C_n, \hat{b}), \\
  [\ell] & \mapsto [\widetilde{m} \cdot \ctransf(\ell) \cdot
  \widetilde{m}^{-1}],
\end{align*}
where $\hat{b}$ is a lift of the base point $b$ of
$\pi_1(\knotexterior)$, $\ell$ denotes a closed loop in $C_n$ and
$\widetilde{m}$ is a lift of the meridian $\mu$ with the initial point
$\hat{b}$.  Let $\homPiICnEk \co \pi_1(C_n,\hat{b}) \to
\pi_1(\knotexterior,b)$ be the homomorphism induced by the
projection from $C_n$ to $\knotexterior$.  Then we denote by
$\mu_n$ a meridian of $C_n$ such that $\homPiICnEk(\mu_n) = \mu^n$.
Note that the covering transformation does not preserve base point
$\hat{b}$.

\begin{remark}
  It follows from the definition of $\tau$ that
  \begin{enumerate}
  \item for every element $\gamma$ in $\pi_1(C_n)$, the action of
    $\tau^n$ sends $\gamma$ to $\mu_n \,\gamma\, \mu_n^{-1}$ where
    $\tau^k$ denotes $k$-times composition of $\tau$; and
  \item the homomorphism $\tau$ sends $\mu_n$ to itself in
    $\pi_1(C_n)$.
  \end{enumerate}
\end{remark}

Lin presentations are useful to describe the fundamental groups of the
cyclic cover $C_n$ over $\knotexterior$.  Indeed, we can construct
$C_n$ by cutting $\knotexterior$ along a Seifert surface and gluing
$n$ copies. Then every closed loop on the Seifert surface can be
lifted to $C_n$.
\begin{lemma}\label{lemma:presentation_for_C_n}
Given a Lin presentation of $\pi_1(E_K)$ of the form 
  \[
  \pi_1(E_K) = \langle x_1, \ldots, x_{2g}, \mu \,|\, \mu \alpha_i \mu^{-1} = \beta_i, i=1, \ldots,2g \rangle,
  \] 
  then $\pi_1(C_n)$ admits a presentation of the form 
  \begin{equation}\label{eqn:presentation_for_C_n}
    \pi_1(C_n) = 
    \left\langle 
      \begin{array}{c}
        \tilde{x}_1, \ldots, \tilde{x}_{2g}, \\ 
        \tau^j \tilde{x}_1, \ldots, \tau^j \tilde{x}_{2g}, \mu_n 
      \end{array}
      \, \left| \, 
        \begin{array}{c}
          \mu_n \tilde{\alpha}_i^{(0)} \mu_n^{-1} = \tilde{\beta}_i^{(n-1)}, \,
          \tilde{\alpha}_i^{(j)} = \tilde{\beta}_i^{(j-1)}, \\
          1 \leq i \leq 2g,\,
          1 \leq j \leq n-1 
        \end{array}
      \right.
    \right\rangle
  \end{equation}
  where $\tilde{x}_i$ is the lift of $x_i$ to $C_n$ and 
  $\tilde{\alpha}_i^{(j)}$, $\tilde{\beta}_i^{(j)}$ denote the words obtained from $\alpha_i$, $\beta_i$ 
  by replacing $x_1, \ldots, x_{2g}$
  with $\tau^j \tilde{x}_1, \ldots, \tau^j \tilde{x}_{2g}$ for 
  $i=1, \ldots, 2g$ and $j=0, \ldots, n-1$.
\end{lemma}

Note that $x_i$ in $\knotexterior$ does not intersect with the free
Seifert surface.  So we can take a closed loop in $C_n$ as a lift.

\begin{proof}
  The presentation $(\ref{eqn:presentation_for_C_n})$ follows from the
  construction of $C_n$ by cutting $\knotexterior$ along the free
  Seifert surface $S$ and gluing $n$ copies of it along their boundary surfaces 
  in an appropriate manner. First we separate $\Sigma$ into two
  handlebodies $H_1 = S \times [-1, 1]$ and $H_2$ as in
  Lemma~\ref{lemma:preferred_Heegaard_splitting} and take a simple
  closed curve $K'$ on $S \times \{0\}$ parallel to $K$ by pushing $K$
  slightly into $H_1$.  We denote by $S'$ a compact surface obtained
  by shrinking $S$ so that the boundary coincides with $K'$.  Cutting
  $H_1$ along $S'$ and gluing $n$ copies, we have the cyclic cover
  $\wtilde{H}_1$ of $H_1$.  Since the cyclic cover $\wtilde{H}_1$ is
  deformed to the union $T^2$ and $n$ copies of $S'$, the fundamental
  group $\pi_1(\wtilde{H}_1)$ is generated by the homotopy classes of
  all lifts of $a_i$ $(1\leq i \leq 2g)$ and $\mu_n$ where $a_1,
  \ldots, a_{2g}$ are circles in a spine of $S'$.  Next gluing
  $n$-copies of $H_2$ to $\wtilde{H}_1$ as in
  Figure~\ref{fig:decomp_C_n}, we obtain the $n$--fold cyclic cover of
  $\knotexterior$ and the presentation
  $(\ref{eqn:presentation_for_C_n})$ by van Kampen's theorem. 
\end{proof}

\begin{figure}[!ht]
  \begin{center}
    \includegraphics[scale=.6]{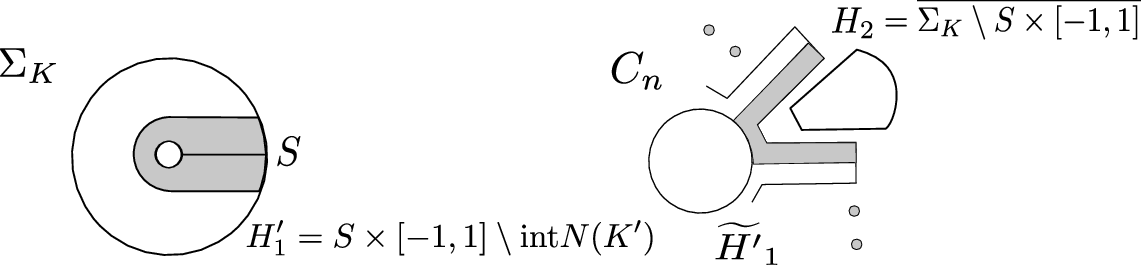}
  \end{center}
  \caption{Decomposition of $C_n$}\label{fig:decomp_C_n}
\end{figure}

The presentation of $\pi_1(C_n)$ in Lemma
\ref{lemma:presentation_for_C_n} induces a presentation of
$H_1(C_n;\Z)$ as follows.  It is known that this presentation of
$H_1(C_n;\Z)$ can be also proved by using a Mayer-Vietoris argument and
the Alexander duality, see~\cite[Chapter 9]{lickorish97:_introd_to_knot_theory}.
\begin{lemma}\label{lemma:presentation_homology_C_n}
  We keep the notations in Proposition~\ref{prop:general_U_V} and
  Lemma \ref{lemma:presentation_for_C_n}.  Let $Z$ be a free abelian
  group generated by $2ng$ elements as follows:
  \[
  Z := \left(\oplus_{i=1}^{2g}\Z [\wtilde{x}_i]\right) \oplus
  \left(\oplus_{i=1}^{2g}\Z [\tau(\wtilde{x}_i)]\right) \oplus \cdots \oplus
  \left(\oplus_{i=1}^{2g}\Z [\tau^{n-1}(\wtilde{x}_i)]\right).
  \]
  Then the presentation of $H_1(C_n;\Z)$ is given by
  \[
  Z \xrightarrow{ 
      \begin{pmatrix}
        A\\
        {\small 0 \cdots 0}
      \end{pmatrix}
     } Z\oplus \Z[\mu_n] \to H_1(C_n;\Z),
  \]
  where $A$ is the $(2ng \times 2ng)$-matrix given by 
  \[
  A= \left(
    \begin{array}{ccccc}
      V & -U &\cdots & 0 & 0  \\
      0 & V &-U &\cdots & 0 \\
      0 & 0 & V & \cdots & 0 \\
      \vdots & \ddots & \ddots& \ddots& \vdots \\
      -U & \cdots & 0 & 0 & V   
    \end{array}
  \right).
  \]
  Moreover if we set 
  $\bm{x}_j = \trans{([\tau^j(\wtilde{x}_1)], \ldots, [\tau^j(\wtilde{x}_{2g})])}$, 
  then we have the following relation in $H_1(C_n;\Z)$:
  \[
  \bm{x}_0 + \cdots + \bm{x}_{n-1} = \bm{0}.
  \]
\end{lemma}
\begin{proof}
  The above presentation is given by the abelianization of 
  the fundamental group as in Lemma~\ref{lemma:presentation_for_C_n}. So we omit
  the details. Here we show the relation $\bm{x}_0 + \cdots +
  \bm{x}_{n-1} = \bm{0}$.  From the presentation of $H_1(C_n;\Z)$, we
  have the following relation:
  \[
  \left\{
    \begin{array}{rrrrcc}
      V\bm{x}_0  &           &              &-U\bm{x}_{n-1} &=& \bm{0} \\
      -U\bm{x}_0 &+V\bm{x}_1 &              & &=& \bm{0} \\
      &           &\vdots        & & & \vdots\\
      &           &-U\bm{x}_{n-2}&+V\bm{x}_{n-1} &=& \bm{0}.
    \end{array} 
  \right.
  \]
  By taking sum on the both side, we have
  \[
  (V - U)(\bm{x}_0 + \cdots + \bm{x}_{n-1}) = \bm{0}.
  \]
  By Proposition~\ref{prop:general_U_V}
  (cf. Remark~\ref{remark:Alexander_poly}) and the fact that the
  Alexander polynomial of a knot at $t=1$ always equals $\pm 1$, we see that
  $\det (V - U) = \Delta_K(1) = \pm 1$, which implies that $\bm{x}_0 +
  \cdots + \bm{x}_{n-1} = \bm{0}$. 
\end{proof}

\subsection{The induced presentation for the fundamental group of 
the $n$--fold branched cover}
\label{subsection:pres_piI_Bn}
The main result in this subsection is Lemma~\ref{lemma:presentation_for_Sigma_n}, 
which describes the presentation for $\pi_1(\branchedcoverZHS{n})$ 
induced by a choice of Lin presentation for $\pi_1(\knotexterior)$. 
%We describe the presentation for $\pi_1(\branchedcoverZHS{n})$ induced
%by a Lin presentation of $\pi_1(\knotexterior)$.  

We denote by $j_*$ the homomorphism from $\pi_1(C_n)$ to
$\pi_1(\branchedcoverZHS{n})$ induced by the inclusion from $C_n$ onto
$\branchedcoverZHS{n}$.  Now we have the following diagrams on the
knot exterior and its covering spaces:
\[
\xymatrix{
  C_n \ar[r]^{j} \ar[d]_{p} & \branchedcoverZHS{n} \ar[d]^{q}\\
  \knotexterior \ar[r]_{i} & \Sigma, } 
\qquad 
\xymatrix{
  \pi_1(C_n) \ar[r]^{j_*} \ar[d]_{p_*} & \pi_1(\branchedcoverZHS{n}) \ar[d]^{q_*}\\
  \pi_1(\knotexterior) \ar[r]_{i_*} & \pi_1(\Sigma).  }
\]

Since $\branchedcoverZHS{n}$ is obtained from $C_n$ by gluing $D^2 \times S^1$ 
along their boundariesand the universality of the amalgamated product, we have
the following diagram:
\begin{equation}\label{diag:Kampen}
  \xymatrix@R=0pt{
    &\pi_1(D^2 \times S^1) \ar[dr]& & \\
    \pi_1(T^2) \ar[ur] \ar[dr]& &\pi_1(D^2 \times S^1) 
    \mathop{\ast}_{\pi_1 (T^2)} \pi_1(C_n) \ar[r]^(0.7){\isomInVanKampen} 
       & \pi_1(\branchedcoverZHS{n}). \\
    &\pi_1(C_n) \ar[ur] & &
  }
\end{equation}
Van Kampen's theorem implies that the above map $\isomInVanKampen$ is an isomorphism. The
next result is a basic fact in group theory.
\begin{lemma}
Suppose $\varphi \co H \to G_1$ is an epimorphism and the following diagram of groups is commutative:
  \[
  \xymatrix@R=0pt{
    &G_1 \ar[dr]^{i_{G_1}}& \\
    H \ar[ur]^{\varphi} \ar[dr]_{\psi}& &G_1 *_H G_2.\\
    &G_2 \ar[ur]_{i_{G_2}}& }
  \]
  Then $i_{G_2}$ is also surjective and the kernel of $i_{G_2}$ is given by
  $\nclosure{\psi(\ker \varphi)}$ where $\nclosure{T}$ is the normal closure of a subset
  $T$.  In particular, we have the isomorphism 
  $G_2 / \nclosure{\psi(\ker \varphi)} \simeq G_1 *_H G_2$.
\end{lemma}
Since the homomorphism $\pi_1(T^2) \to \pi_1(D^2\times S^1)$ is
surjective in the diagram~$(\ref{diag:Kampen})$ and the kernel is
generated by a meridian, we have the isomorphism:
\begin{equation}\label{eqn:isom_C_n_Sigma2}
  \pi_1(C_n) / \nclosure{\mu_n} \simeq \pi_1(\branchedcoverZHS{n}),
\end{equation}
where $\mu_n$ is the meridian such that $\homPiICnEk (\mu_n) = \mu^n$.
Hence we can regard the homomorphism $\hat{\pi}$ as the projection
from $\pi_1(C_n)$ onto $\pi_1(C_n) / \nclosure{\mu_n}$.
Similarly, we have the following isomorphism for $\pi_1(\Sigma)$:
\begin{equation}\label{eqn:isom_SigmaK_Sigma}
  \pi_1(\knotexterior)/\nclosure{\mu} \simeq \pi_1(\Sigma).
\end{equation}

From Lemmas~\ref{lemma:presentation_for_C_n} and
\ref{lemma:presentation_homology_C_n} and the above relation between
$\pi_1(C_n)$ and $\pi_1(\branchedcoverZHS{n})$, we obtain the
following presentations of $\pi_1(\branchedcoverZHS{n})$ and
$H_1(\branchedcoverZHS{n};\Z)$.
\begin{lemma}\label{lemma:presentation_for_Sigma_n}
  Given a Lin presentation of $\pi_1(\knotexterior)$ of the form
  \[
  \pi_1(\knotexterior) = 
  \langle 
    x_1, \ldots, x_{2g}, \mu \,|\, 
    \mu \alpha_i \mu^{-1} = \beta_i , 
    i=1, \ldots, 2g
  \rangle,
  \]
  then 
  $\pi_1(\branchedcoverZHS{n})$ admits a presentation of the form 
  \begin{equation}\label{eqn:presentation_for_Sigma_n}
    \pi_1(\branchedcoverZHS{n}) = 
    \langle
    \tau^j \tilde{x}_1, \ldots, \tau^j \tilde{x}_{2g} \,|\,
    \tilde{\alpha}_{i}^{(j)} = \tilde{\beta}_i^{(j-1)},
    1 \leq i \leq 2g, 1 \leq j \leq n
    \rangle, 
  \end{equation}
  where $\tilde{x}_i$ is the lift of $x_i$ to $B_n$ and $\tilde{\alpha}_i^{(j)}$, 
  $\tilde{\beta}_i^{(j)}$
  denote the words obtained from $\alpha_i$, $\beta_i$ 
  by replacing $x_1, \ldots, x_{2g}$ with
  $\tau^j \tilde{x}_1, \ldots, \tau^j \tilde{x}_{2g}$
  for $i=1, \ldots, 2g$ and $j=1, \ldots, n$.  The homology group
  $H_1(\branchedcoverZHS{n};\Z)$ is expressed as
  \[
  Z \xrightarrow{A} Z \to H_1(\branchedcoverZHS{n};\Z),
  \]
  where $Z$ and $A$ are as in
  Lemma~\ref{lemma:presentation_homology_C_n}.  Moreover for 
  $\bm{x}_j = \trans{([\tau^j(\wtilde{x}_1)], \ldots, [\tau^j(\wtilde{x}_{2g})])}$, 
  the following relation holds in
  $H_1(C_n;\Z)$:
  \[
  \bm{x}_0 + \cdots + \bm{x}_{n-1} = \bm{0}.
  \]
\end{lemma}
Note that it holds that 
$\tau^n \tilde x_i = \mu_n \tilde x_i \mu_n^{-1}
= \tilde x_i$ in $\pi_1(\branchedcoverZHS{n})$.

\begin{remark}\label{remark:RBn_in_RCn}
  Since the group $\pi_1(\branchedcoverZHS{n})$ is just the quotient
  of $\pi_1(C_n)$ by the relation $\mu_n=1$, we can view 
  $R(\branchedcoverZHS{n})$ as consisting of representations
  $\rho \in R(C_n)$ satisfying $\rho(\mu_n) = \I$.
\end{remark}

\section{On the characters of metabelian representations}
\label{section:char_metabelian}
The purpose of this section is to describe a feature of
metabelian characters via $\Z_2$--action (involution) on character varieties.
Subsections~\ref{subsection:reducible_representations} and
\ref{subsection:irreducible_metabelian} deal with reducible metabelian
representations and irreducible ones separately.  We will see that the
metabelian characters gives the fixed points on the
character variety via the involution in 
Subsection~\ref{subsection:involution_character_variety}.

\subsection{Reducible representations}
\label{subsection:reducible_representations}
From the definition, every reducible representation is conjugate to a
representation whose image consists of upper triangular matrices in
$\SL$.
\begin{remark}\label{remark:metabelian}
  If $A$ and $B$ are upper triangular $\SL$-elements, then the
  commutator $[A, B]$ lies in ${\rm Para}$ with eigenvalues $1$.  So
  we can see that the commutator subgroup of the image of any
  reducible representation is an abelian subgroup in $\SL$. Hence all
  reducible representations are metabelian.  The set
  $R^{\mathrm{meta}}(\knotexterior)$ of metabelian representations can
  be decomposed as a union
  \[
  \{\rho \in R^{\mathrm{irr}}(\knotexterior) \,|\, 
    \rho:{\rm metabelian}\} 
  \cup R^{\mathrm{red}}(\knotexterior).
  \]
\end{remark}

Furthermore we have two kinds of reducible representations, one is
abelian and another is non--abelian.  Every abelian representation
factors through the abelianization $H_1(\knotexterior;\Z)$.  Since
the homology group $H_1(\knotexterior;\Z)$ is generated by the
homology class of the meridian $\mu$, each abelian representation is
determined by the image of the meridian $\mu$.  
For a given $\lambda \in \C^*$, we denote by $\varphi_{\lambda}$ the 
abelian representation given by
\[
\varphi_{\lambda} \co \pi_1(\knotexterior)   \to \SL, \quad
\mu \mapsto 
  \begin{pmatrix}
    \lambda & 0 \\
    0 & \lambda^{-1}
  \end{pmatrix}.
\]
By taking conjugation, we can assume that the image of any reducible
representation consists of upper triangular $\SL$-elements.  This
means that for any reducible representation $\rho$, there is a complex
number $\lambda$ such that the character of $\rho$ is the same as that
of the abelian representation $\varphi_\lambda$.
For non--abelian reducible
representations, it is known that the following fact holds.
\begin{lemma}[\cite{CCGLS}, \cite{heusener01:_defor_of_reduc_repres_of}]
\label{lemma:reducible_rep}
  There exists a reducible non--abelian representation
  \[\rho \co \pi_1(\knotexterior) \to \SL\] such that
  $\chi_{\rho}=\chi_{\varphi_{\lambda}}$ if and only if
  $\Delta_K(\lambda^2)=0$.
\end{lemma}
This is a well--known result of Burde~\cite{burde67:_darst_von_knoten}
and de~Rham~\cite{rham67:_introd_aux_polyn_dun_noee_ud} if
$\knotexterior$ is the knot exterior of a knot in $S^3$.

\subsection{The characters of metabelian and binary dihedral representations}
\label{subsection:irreducible_metabelian}
In this subsection, we focus on irreducible metabelian representations
of a knot group.  We begin with a review on binary dihedral representations of a
knot group $\pi_1(\knotexterior)$.  As we will see, binary dihedral
representations form a subset of metabelian representations and give
the representatives of conjugacy classes for irreducible metabelian
$\SL$-representations.

First we recall the binary dihedral group in $\SU$. Set
\[
S^1_A = 
\left\{\left.
  \begin{pmatrix}
    a & 0 \\
    0 & \bar{a}
  \end{pmatrix}
  \,\right|\,
  a \in \C,|a|=1 
\right\}
\quad
\text{and}
\quad 
S^1_B = 
\left\{\left.
  \begin{pmatrix}
    0 & b \\
    -\bar{b} & 0
  \end{pmatrix}
  \,\right|\,
  b \in \C,
  |b|=1
\right\}
\]
and define the binary dihedral group as $N = S^1_A \cup S^1_B \subset
\SU$. It is easy to see $N$ is closed under multiplication, and a
simple calculation reveals that for any $g$, $h \in N$, we have $[g,
h] \in S^1_A$. 
An $\SL$-representation is said to be {\it binary dihedral} 
if the image is contained in $N$.
It follows that $[N, N] = S^1_A$ is abelian, that
$N$ is a metabelian group, and that every $\SL$-representation
$\rho$ with image in $N$ is metabelian.
This shows that every binary dihedral representation is metabelian,
and the next result gives an equivalence between irreducible
metabelian characters and irreducible binary dihedral characters.

\begin{proposition}\label{prop:meta_conj_dihedral}
  Every irreducible metabelian $\SL$-representation of
  $\pi_1(\knotexterior)$ is conjugate to an irreducible binary
  dihedral representation.  Hence we have the following equation:
  \[
  \{\chi_\rho\,|\, \rho:{\rm irreducible\, metabelian}\} =
  \{\chi_\rho\,|\, \rho:{\rm irreducible\, binary\, dihedral}\}.
  \]
  Moreover the number of conjugacy classes of irreducible metabelian
  representations is given by
  \[
  \frac{|\Delta_K(-1)|-1}{2},
  \]
  where $\Delta_K(t)$ is the Alexander polynomial of $K$.
\end{proposition}

For a knot in $S^3$, we can find the formula 
for the number of conjugacy classes of irreducible 
metabelian representations in \cite[Proposition 1.1 and Theorem~1.2]
{nagasato07:_finit_of_section_of_sl}.  
Eric Klassen~\cite[Theorem~10]{Klassen:1991} also proved the same formula 
for binary dihedral representations. 
Hence Proposition~\ref{prop:meta_conj_dihedral} for a knot in $S^3$ 
can be deduced by combining these two results. 
\begin{proof}[Proof of Proposition~\ref{prop:meta_conj_dihedral}]
  Let $\rho$ be an irreducible and metabelian $\SL$-representation of 
  $\pi_1(\knotexterior)$.
  Then for the generators in a Lin presentation~$(\ref{eqn:Lin_equation})$ 
  of a knot in an integral homology 3--sphere $\Sigma$,
  by taking conjugation, we have the following form of binary dihedral representations:
  \begin{equation}\label{eqn:explicit_binary_dihedral}
  \rho(x_i) = \left(
    \begin{array}{cc}
      \alpha_i & 0 \\
      0 & \alpha_i^{-1}
    \end{array}
  \right), \quad 
  \rho(\mu) = \left(
    \begin{array}{cc}
      0 & 1 \\
      -1 & 0
    \end{array}
  \right).
  \end{equation}
  Then we can apply the same argument as in 
  the proof of \cite[Proposition 1.1 and Theorem~1.2]
  {nagasato07:_finit_of_section_of_sl} 
  to deduce the above form and counting the conjugacy classes. 
\end{proof}

\begin{remark}
  We can find a higher rank analogy of Proposition~\ref{prop:meta_conj_dihedral} 
  for knots in $S^3$ in~\cite{BodenFriedlMetabelianI}.
\end{remark}

\subsection{An involution on the $\SL$-character variety}\label{subsec:slice}
\label{subsection:involution_character_variety}
In this subsection, we introduce involutions $\iota$ and $\hat{\iota}$
on $R(\knotexterior)$ and $X(\knotexterior)$ and identify the fixed point set of
$\hat{\iota}$ on $X^{\mathrm{irr}}(E_K)$ with irreducible metabelian
characters.  For $g \in \pi_1(\Sigma)$, we denote by $[g] \in
H_1(E_K;\Z) \simeq \Z$ the associated element in the homology group,
which we identify with $\Z$ by the isomorphism that sends the meridian
$\mu$ to the generator $1$. Next, we define an involution $\iota$ on
$R(E_K)$ as follows.  Given $\rho \in R(E_K)$, let $\iota(\rho)$ be
the representation such that $\iota(\rho)(g) = (-1)^{[g]}\rho(g)$ for
all $g \in \pi_1(E_K)$.  It is immediate that $\iota$ is an involution
on $R(E_K)$, and it induces an involution $\hat{\iota}$ on the character variety
$X(E_K)$ defined as follows. Given
$\chi_\rho \in X(E_K)$, let $\hiota(\chi_\rho)$ be the character such that
$\hat{\iota}(\chi_\rho) = (-1)^{[g]} \chi_{\rho}(g)$ for any $g \in
\pi_1(E_K)$. It is easy to check that $\hat{\iota}$ is an involution
on $X(E_K)$ and that $\hat{\iota}(\chi_\rho) = \chi_{\iota(\rho)}$
from the commutativity of trace with the scalar multiplication.

The involution $\hiota$ has the following fixed point set.
\begin{proposition}\label{prop:fixed_point}
  The fixed point set of $\hiota$ in $X^{\mathrm{irr}}(\knotexterior)$ is
  the set of irreducible metabelian characters, i.e.,
  \[
  X^{\mathrm{irr}}(\knotexterior)^{\hiota} = 
  \{\chi_\rho \,|\,
  \rho \co \rm{irreducible\, metabelian}\}
  \]
\end{proposition}
We apply the following elementary lemma to prove
Proposition~\ref{prop:fixed_point}.

\begin{lemma}\label{lemma:elementary_calc}
  We set $A \in M_2(\C)$ and 
  $C=\begin{pmatrix}
      0 & 1 \\
      -1 & 0
    \end{pmatrix}.$
Then we have 
  \[\trace AC = 0 \Leftrightarrow A=\trans{A}.\]
\end{lemma}

\begin{proof}[Poof of Proposition~\ref{prop:fixed_point}]
  We choose a Lin presentation for the knot group:
  \begin{equation}\label{eqn:pres_Lin}
    \pi_1(\knotexterior)=
    \langle x_1, \ldots, x_{2g}, \mu \,|\, \mu\alpha_i\mu^{-1} = \beta_i, \, i=1, \ldots, 2g \rangle.
  \end{equation}

  First we show that for every irreducible metabelian
  $\SL$-representation $\rho$ the image $\iota(\rho)$ is conjugate to
  $\rho$.
  It follows from Proposition~\ref{prop:meta_conj_dihedral} that after
  taking suitable conjugation, the generators in the presentation
  $(\ref{eqn:pres_Lin})$ are sent to the following matrices by $\rho$:
  \begin{equation}\label{eqn:irred_binary_form}
    \rho(x_i) = 
    \left(
      \begin{array}{cc}
        \lambda_i & 0 \\
        0 & \lambda_i^{-1}
      \end{array}
    \right), 
    \quad
    \rho(\mu) =
    \left(
      \begin{array}{cc}
        0 & 1 \\
        -1 & 0
      \end{array}
    \right).
  \end{equation}
  Since every $x_i$ is contained in the commutator subgroup of
  $\pi_1(\knotexterior)$, the images of these generators by
  $\iota(\rho)$ are given by
  \[
  \iota(\rho)(x_i) = \left(
    \begin{array}{cc}
      \lambda_i & 0 \\
      0 & \lambda_i^{-1}
    \end{array}
  \right), \quad 
  \iota(\rho)(\mu) = -\left(
    \begin{array}{cc}
      0 & 1 \\
      -1 & 0
    \end{array}
  \right).
  \]
  Hence we have
  \[
  \iota(\rho) = P \rho P^{-1}, \quad 
  P = \left(
    \begin{array}{cc}
      \sqrt{-1} & 0 \\
      0 & -\sqrt{-1}
    \end{array}
  \right).
  \]
  In particular, we obtain the equation 
  $\chi_\rho = \chi_{\iota(\rho)}$.

  Next we show that an irreducible $\SL$-representation $\rho$
  satisfying that $\chi_\rho = \chi_{\iota(\rho)}$ is metabelian.  From the
  equations $\chi_{\iota(\rho)}(\mu) = -\chi_\rho(\mu)$ and $\chi_\rho
  = \chi_{\iota(\rho)}$, we have $\trace\rho(\mu) = 0$.  Therefore by
  taking conjugation we can assume that $\rho(\mu)$ is given by the
  matrix 
  $\begin{pmatrix} 
    0 & 1 \\ 
    -1 & 0 
  \end{pmatrix}$.
  We choose any two elements $\gamma_1, \gamma_2$ from
  $[\pi_1(\knotexterior), \pi_1(\knotexterior)]$.
  Since $[\gamma_i\mu]=1$, we have that
  $\chi_{\iota(\rho)}(\gamma_i\mu) = -\chi_\rho(\gamma_i\mu)$.  From
  the assumption $\chi_\rho=\chi_{\iota(\rho)}$, we obtain $\trace
  \rho(\gamma_i\mu) = 0$.  By Lemma \ref{lemma:elementary_calc},
  $\rho(\gamma_i) = \trans{\rho(\gamma_i)}$
  holds for $i=1, 2$. Since $\gamma_1 \gamma_2 \in [\pi_1(\knotexterior), \pi_1(\knotexterior)]$,
  we also have $\rho(\gamma_1\gamma_2) = \trans{\rho(\gamma_1\gamma_2)}$
  Then we have the following equalities of matrices:
  \[
  \rho(\gamma_1\gamma_2) 
  = \trans{\rho(\gamma_1\gamma_2) }
  = \trans{\rho(\gamma_2)}\trans{\rho(\gamma_1)} 
  = \rho(\gamma_2)\rho(\gamma_1)
  = \rho(\gamma_2\gamma_1).
  \]
  Therefore the representation $\rho$ is metabelian. 
\end{proof}

By Propositions~\ref{prop:meta_conj_dihedral} and
\ref{prop:fixed_point}, we can conclude the following corollary.
\begin{corollary}\label{cor:fixed_point_dihedral}
  \[
  X^{\mathrm{irr}}(\knotexterior)^{\hiota} = 
  \{ \chi_\rho \,|\, 
  \rho \co {\rm irreducible\, binary\, dihedral}\}.
  \]
\end{corollary}

The fixed point set $X(\knotexterior)^{\hiota}$ could also contain
reducible characters. Actually, as we shall see, this fixed point set
$X(\knotexterior)^{\hiota}$ contains only one character coming from
reducible representations.

Before giving this proof, we define subsets in the
character variety where the fixed point set
$X(\knotexterior)^{\hiota}$ lives.  Every character in such a
subset sends the meridian $\mu$ to the same value.  We define such
subsets as a level set of a function on the character variety.

For any given $\gamma \in \pi_1(\knotexterior)$, we denote by
$I_{\gamma}$ a function on $X(\knotexterior)$ defined as $ I_\gamma
\co X(\knotexterior) \to \C,\, \chi \mapsto \chi(\gamma).  $ This
function is called the trace function of $\gamma$.

\begin{definition}\label{def:slice}
  We define the subset $S_c(\knotexterior)$ in $X(\knotexterior)$
  as the level set of $I_{\mu}$ at $c$, i.e.,
  \[
  S_c(\knotexterior) = I_{\mu}^{-1}(c).
  \]
  This subset is called the slice with the trace of meridian fixed by $c$.
\end{definition}
We also set $R_c(\knotexterior)$ as 
$R_c(\knotexterior) := \{\rho \in R(\knotexterior)\,|\, \trace \rho(\mu) = c\}.$

\begin{lemma}\label{lemma:fixed_pt_in_S0}
  The fixed point set $X(\knotexterior)^{\hiota}$ is contained in
  $S_0(\knotexterior)$ and the restriction of $\hiota$ to $S_0(\knotexterior)$ gives an
  involution on it.
\end{lemma}
\begin{proof}
  We have seen that every character $\chi$ in
  $X(\knotexterior)^{\hiota}$ satisfies that $\chi(\mu)=0$ in the
  proof of Proposition \ref{prop:fixed_point}.  For any element $\chi$
  in $S_0(\knotexterior)$, we have
  \[
  \hiota(\chi)(\mu) = - \chi(\mu) = 0.
  \]
  The image of $S_0(\knotexterior)$ by $\hiota$ is contained in
  $S_0(\knotexterior)$ itself.  
\end{proof}
\begin{remark}\label{remark:interchange_slice}
  The involution $\iota$ interchanges $R_c(\knotexterior)$ and
  $R_{-c}(\knotexterior)$, and likewise $\hiota$ interchanges the
  $S_c(\knotexterior)$ and $S_{-c}(\knotexterior)$.
\end{remark}
Every representation in $R_0(\knotexterior)$ maps the meridian
$\mu$ to matrices whose eigenvalues are $\pm \sqrt{-1}$.  By Lemma
\ref{lemma:reducible_rep} and $\Delta_K(-1) \not = 0$, we have no
non--abelian reducible representation in $R_0(\knotexterior)$.
\begin{remark}\label{remark:reducible_reps_in_S0}
  Every reducible representation in $R_0(\knotexterior)$ is abelian.
\end{remark}
Therefore, by Lemma \ref{lemma:fixed_pt_in_S0}, the set of reducible
characters in $X(\knotexterior)^{\hiota}$ consists of only one
character of the abelian representation $\varphi_{\sqrt{-1}}$.
Summarizing, we have the following proposition.
\begin{proposition}\label{prop:fixed_point_set}
  The fixed point set $X(\knotexterior)^{\hiota}$ is expressed as
  \[
  X(\knotexterior)^{\hiota} = S_0(\knotexterior)^{\hiota} =
  \{\chi_{\rho}\,|\, \rho:{\rm irreducible\, binary\, dihedral} \} \cup
  \{\chi_{\varphi_{\sqrt{-1}}}\}.
  \]
\end{proposition}

\section{Maps between representation spaces} 
\label{section:map_rep_spaces}
In this section, we derive how to make an $\SL$-representation of
$\pi_1(\branchedcoverZHS{n})$ from that of $\pi_1(\knotexterior)$ 
via the maps between representation spaces defined 
by the pull--backs of the homomorphisms $\homPiICnEk$, $q_*$, $i_*$ and $j_*$ 
between the fundamental groups. 
These maps satisfy the following diagram among the $\SL$-representation spaces of
$\pi_1(\knotexterior)$, $\pi_1(C_n)$, $\pi_1(\Sigma)$ and
$\pi_1(\branchedcoverZHS{n})$:
\[
\xymatrix{
  \pi_1(C_n) \ar[r]^{j_*} \ar[d]_{p_*} & \pi_1(\branchedcoverZHS{n}) \ar[d]^{q_*}\\
  \pi_1(\knotexterior) \ar[r]_{i_*} & \pi_1(\Sigma),  }
\qquad
\xymatrix{
  R(C_n)  & \ar[l]_{\pullbackj} R(\branchedcoverZHS{n})  \\
  R(\knotexterior) \ar[u]^{\pullbackp}& \ar[l]^{\pullbacki} \ar[u]_{\pullbackq}
  R(\Sigma).}
\]
Note that $\pullbackj$ is injective since $j_*$ is surjective. 
As in Remark~\ref{remark:RBn_in_RCn}, 
we can regard $\pullbackj(R(\branchedcoverZHS{n}))$ as the subset in
$R(C_n)$ consisting of $\rho \in R(C_n)$ which sends $\mu_n$ to the identity matrix.  
Hereafter we simply identify $R(\branchedcoverZHS{n})$ with its image under $\pullbackj$,
and likewise regard $R(\Sigma)$ as a subset of $R(\knotexterior)$.

In this setting, $\SL$-representations of $\pi_1(\branchedcoverZHS{n})$ 
is constructed from the following subset of those of $\pi_1(\knotexterior)$:
\begin{equation}\label{set:domain_Phi}
  \{\rho \in R(\knotexterior)\,|\, \rho(\mu)^n = \I\} \cup 
  \{\rho \in R(\knotexterior)\,|\, \rho(\mu)^n = -\I\}.
\end{equation}
In fact, our interest lies in the case of $n=2$ because this case has distinct features 
from the other cases (see Section~\ref{section:slice_S0}). 
To compare their features, we deal with general definitions including 
the cases other than $n=2$ in this section. 

We start with the observation that the pull--back of $\rho$ in $R(E_K)$ 
satisfying $\rho(\mu)^n = \I$ defines an $\SL$-representation of 
$\pi_1(\branchedcoverZHS{n})$ however this correspondence covers very few range
of $R(\branchedcoverZHS{n})$.
To deal with more range, we consider the pair of subset as in~$(\ref{set:domain_Phi})$.
We first observe the intersection $\pullbackp (R(\knotexterior)) \cap
R(\branchedcoverZHS{n})$ and the induced correspondence from the
following subset of $R(\knotexterior)$ to $R(\branchedcoverZHS{n})$. 
\begin{lemma}\label{lemma:intersection_ordinary}
  \begin{equation}\label{eqn:intersection_ordinary}
    (\pullbackp)^{-1} ( R(\branchedcoverZHS{n}) )
    =\{\rho \in R(E_K) \,|\,
      \rho(\mu)^n = \I\}. 
  \end{equation}
\end{lemma}
\begin{proof}
  It follows that $\rho \in (\pullbackp)^{-1} (R(\branchedcoverZHS{n}))
  \Leftrightarrow \pullbackp \rho (\mu_n) = \bm{1} \Leftrightarrow
  \rho(\mu^n) = \bm{1}$.  
\end{proof}
From the intersection $(\ref{eqn:intersection_ordinary})$ in
Lemma~\ref{lemma:intersection_ordinary}, we can make a correspondence
from the subset $\{\rho \in R(E_K) \,|\, \rho(\mu)^n = \I\}$ into 
$R(\branchedcoverZHS{n})$ by the pull--back $\pullbackp$,
i.e., 
\[
\pullbackp \co
\{\rho \in R(E_K) \,|\, \rho(\mu)^n = \I\} \to R(\branchedcoverZHS{n}) (\subset R(C_n)).
\]

\begin{remark}
In the special case that $n=2$, 
the subset
$(\pullbackp)^{-1} (R(\branchedcoverZHS{2}))$ turns
into
\[
\{\rho \in R(\knotexterior) \,|\, \rho \co \pi_1(\knotexterior) \to \{\pm
\bm{1}\}\} 
\]
consisting of only two conjugacy classes, both of whose restrictions to
$\pi_1(\branchedcoverZHS{2})$ are trivial. 
Hence, in the construction of an $\SL$-representation of 
$\pi_1(\branchedcoverZHS{2})$ from that of $\pi_1(\knotexterior)$ 
via the map $\pullbackp$, we can only deal with the slices 
$R_{\pm 2}(\knotexterior)$ at $\pm 2$, which correspond to the second roots of unity.
\end{remark}

Next we see that the pull--back by $\pullbackp$ of the
subset~$(\ref{set:domain_Phi})$ defines a $\PSL$-representation of
$\pi_1(\branchedcoverZHS{n})$. This allows us
to consider much more elements of $R(\knotexterior)$.
\begin{lemma}\label{lemma:rep_into_PSL}
  If $\rho$ is an element in the subset $(\ref{set:domain_Phi})$, then
  the pull-back $\pullbackp\rho$ defines $\PSL$-representation of
  $\pi_1(\branchedcoverZHS{n})$.
\end{lemma}
\begin{proof}
  We can regard $\pi_1(\branchedcoverZHS{n})$ as the quotient
  $\pi_1(C_n) / \nclosure{\mu_n}$.  It is enough to prove that
  $\pullbackp \rho (\mu_n) = \pm \I$.  This follows from direct
  calculations. 
\end{proof}

We remark several results on the lifting problem without a sign refinement.  
For example, G. Burde~\cite{burde90:SU_2_repre_spaces} has shown that binary
dihedral representations form the branched point set of covering map
from $\SU$-representations to $\SO$-representations for a two--bridge
knot.  S. Boyer \& X. Zhang~\cite[Section 3]{BoyerZhangCullerShalenSeminorm} 
and J. Porti \&
M. Heusener~\cite[Section 4]{PortiHeusener:PSL_char_variety} deal with
the lifting problem for $\PSL$-representations of a finitely generated
and presented group in context of the character variety.

We will lift $\PSL$-representations of $\pi_1(\branchedcoverZHS{n})$ 
given in Lemma~\ref{lemma:rep_into_PSL} 
to $\SL$ by using a sign refinement in $R(C_n)$.  
We define an involution $\iota_n$ on $R(C_n)$ for our sign refinement 
in Definition~\ref{def:involution_n}.

\begin{definition}\label{def:involution_n}
  Set $\zeta_{2n} = \exp(\pi \sqrt{-1}/n)$. Given $\rho \in
  R(C_n)$, let $\iota_n(\rho)$ be the representation such that
  \[
  \iota_n(\rho) (g) = (\zeta_{2n})^{p_*[g]} \rho(g)
  \]
  for all $g \in \pi_1(C_n)$, where $p_* \co H_1(C_n;\Z) \to
  H_1(\knotexterior) \simeq \Z$ whose image is $n\Z$.  
\end{definition}
It is immediate that $\iota_n$ is an involution on $R(E_K)$, and it
induces an involution $\hiota_n$ on the character variety $X(C_n)$ 
defined as follows.
\begin{definition}
  Given $\chi_\rho \in X(C_n)$, let $\hiota_n (\chi_\rho)$ be the
  character such that \[\hiota_n(\chi_\rho)(g) = (\zeta_{2n})^{p_*[g]}
  \chi_\rho(g).\]
\end{definition}
It is easy to check that $\hiota_n$ is an involution on $X(C_n)$ and
that $\hiota_n(\chi_\rho) = \chi_{\iota_n(\rho)}$ from the commutativity of 
trace with the scalar multiplication.

Now we construct a map  
\[
\Phi \co 
  \{\rho \in R(\knotexterior)\,|\, \rho(\mu)^n = \pm \I\} \to R(\branchedcoverZHS{n})
\]
as follows. Here the domain of $\Phi$ is the
subset~$(\ref{set:domain_Phi})$, denoted by $D(\Phi)$. 
Since any element of $D(\Phi)$ is contained in the slice $R_{\xi(k/n)}(\knotexterior)$ 
at some $\xi(k/n) = 2\cos (k\pi/n)$, we focus on the slice $R_{2\cos
  (k\pi/n)}(\knotexterior)$. 
Let $\rho$ be an element of
$R_{\xi(k/n)}(\knotexterior)$ such that $\rho(\mu)^n = (-\I)^k$.
Then we define an $\SL$-representation $\Phi(\rho)$ of $\pi_1(C_n)$ by
sending each element $\gamma \in \pi_1(C_n)$ to the following matrix:
\[
\Phi(\rho)(\gamma) = (\iota_n)^k (\pullbackp \rho)(\gamma) =(
(\zeta_{2n})^{p_*[\gamma]} )^k \cdot \pullbackp\rho(\gamma),
\]
where $(\iota_n)^k$ is $k$ times composition of $\iota_n$.  
The well--definedness of $\Phi$ follows from the next proposition.
\begin{proposition}\label{prop:2nth-root}
  For $\rho \in R_{\xi(k/n)}(\knotexterior)$ satisfying
  $\rho(\mu)^n = (-\I)^k$, we have the equality 
  $\Phi(\rho)(\mu_n) = \I$.  
  In particular, this induces a homomorphism from
  $\pi_1(C_n) / \nclosure{\mu_n}$ into $\SL$.
\end{proposition}
\begin{proof}
  This proposition follows from the following direct calculation.  Let
  $\rho \in R_{\xi(k/n)}(\knotexterior)$ satisfy 
  $\rho(\mu)^n = (-\I)^k$.
  \begin{align*}
    \Phi(\rho)(\mu_n)
    &= ( (\zeta_{2n})^{p_*[\mu_n]})^k \cdot \pullbackp \rho (\mu_n) \\
    &= (\zeta_{2n})^{kn} \cdot \rho(\mu)^n \\
    &= (-1)^k (-\I)^k \\
    &= \I.
  \end{align*}
  Since $\rho(\nclosure{\mu_n}) = \I$, 
  this representation
  $\Phi(\rho)$ induces a homomorphism from 
  $\pi_1(C_n) / \nclosure{\mu_n} =
  \pi_1(\branchedcoverZHS{n})$ into $\SL$. 
\end{proof}

We note that 
the $\SL$-representation $\Phi(\rho)$ is a lift of the
$\PSL$-representation $\pullbackp \rho$ of
$\pi_1(\branchedcoverZHS{n})$:
\[
\xymatrix{
  & \SL \ar[d]\\
  \pi_1(\branchedcoverZHS{n}) \ar[r]_{\pullbackp \rho}
  \ar[ur]^{\Phi(\rho)} &\PSL .}
\]
The domain
$D(\Phi)$ is decomposed as
\begin{align}
  D(\Phi)
  &= \{\rho \in R(\knotexterior) \,|\, \rho(\mu)^n = \pm \I\} \label{eqn:domain_rep} \\
  &= \bigcup_{k=1}^{n-1} \{\rho \in R_{\xi(k/n)}(\knotexterior)
  \,|\, \rho(\mu)^n = \pm \I\} \cup \{\rho \in
  R(\knotexterior)\,|\,\rho(\mu) = \pm \bm{1}\}.\nonumber
\end{align}
\begin{remark}\label{remark:domain}
  The set $\{\rho \in R(\knotexterior)\,|\, \rho(\mu) = \bm{1}\}$
  coincides with $\pullbacki(R(\Sigma))$.  Since the pull-back
  $\pullbacki$ is injective, we identify $\pullbacki(R(\Sigma))$ with
  $R(\Sigma)$.  The involution $\iota$ interchanges the two sets
  $\{\rho \in R(\knotexterior)\,|\, \rho(\mu) = \bm{1}\}$ and
  $\{\rho \in R(\knotexterior)\,|\, \rho(\mu) = -\bm{1}\}$ with
  each other.  So we can rewrite the domain $D(\Phi)$ in
  $(\ref{eqn:domain_rep})$ as
  \[
  \bigcup_{k=1}^{n-1} \{\rho \in R_{\xi(k/n)}(\knotexterior) \,|\,
  \rho(\mu)^n = \pm \I\} \cup R(\Sigma) \cup \iota(R(\Sigma)).
  \]
\end{remark}

The correspondence $\Phi$ induces a map $\hPhi$ on the subset
$t(D(\Phi))$ in $X(\knotexterior)$
\begin{align*}
\hPhi \co t(D(\Phi)) & \to X(\branchedcoverZHS{n}) \\
\chi_{\rho} & \mapsto  \chi_{\Phi(\rho)}
\end{align*}
It is easy to check the well--definedness from the construction of
$\Phi$ and the commutativity of trace with the scalar multiplication.
The domain $t(D(\Phi))$ of $\hPhi$ is decomposed as
\begin{equation}\label{eqn:domain_Phi}
  t(D(\Phi))=
  \bigcup_{k=1}^{n-1} S_{\xi(k/n)}(\knotexterior) 
  \cup X(\Sigma) \cup \hiota(X(\Sigma))
  =
  \bigcup_{k=0}^n S_{\xi(k/n)}(\knotexterior),
\end{equation}
where we identify $\pullbacki(X(\Sigma))$ with $X(\Sigma)$.  Hence
$\hPhi$ is defined on $S_{\xi(k/n)}(\knotexterior)$ as
\begin{equation}\label{eqn:def_hatPhi}
  \hPhi \co S_{\xi(k/n)}(\knotexterior) \to X(\branchedcoverZHS{n}),\ 
  \chi_\rho \mapsto \chi_{(\iota_n)^k(\pullbackp\rho)}.
\end{equation}

By the definition of $\Phi$, we see the image of any
abelian representation by $\Phi$.

\begin{lemma}\label{lemma:image_abelian}
  For every abelian representation $\rho$ in $D(\Phi)$, the image
  $\Phi(\rho)$ is the trivial representation of
  $\pi_1(\branchedcoverZHS{n})$.
\end{lemma}
\begin{proof}
  Since $\rho$ is abelian, it sends all elements in the commutator
  subgroup of $\pi_1(\knotexterior)$ to the identity matrix.  For
  each generator $\tau^j(\wtilde{x}_i)$ in the presentation of
  Lemma~\ref{lemma:presentation_for_Sigma_n}, the image by
  $\Phi(\rho)$ is expressed as 
  $\Phi(\rho)(\tau^j\wtilde{x}_i) = \rho(\mu^j x_i \mu^{-j}).$ 
  Since all $x_i$ are contained in the
  commutator subgroup of $\pi_1(\knotexterior)$, $\mu^j x_i
  \mu^{-j}$ is also an element in the commutator subgroup.  Hence
  every $\mu^j x_i \mu^{-j}$ is sent to the identity matrix by $\rho$.
  Therefore $\Phi(\rho)$ is trivial. 
\end{proof}

\begin{remark}
  The involutions $\iota$ and $\hiota$ interchange slices as follows:
  \[
  \iota(R_{\xi(k/n)}(\knotexterior)) =
  R_{\xi(1-k/n)}(\knotexterior), \quad
  \hiota(S_{\xi(k/n)}(\knotexterior)) =
  S_{\xi(1-k/n)}(\knotexterior).
  \]
\end{remark}
We investigate the properties of the map $\Phi$ in the rest of this
section.  To consider the image of $\Phi$, we introduce the following
notion and notation concerning the action $\tau$ on $\pi_1(C_n)$.
\begin{definition}
  An $\SL$-representation $\rho$ of $\pi_1(C_n)$ is
  {\it $\tau$--equivariant} if there exists an element $C$ in $\SL$ such
  that the following diagram is commutative:
  \begin{equation}\label{diagram:equiv_diagram}
    \xymatrix{
      \pi_1(C_n) \ar[r]^{\tau} \ar[d]_{\rho} & \pi_1(C_n) \ar[d]^{\rho}\\
      \SL \ar[r]_{Ad_C} & \SL,
    }
  \end{equation}
  where $Ad_C(X):=CXC^{-1}$ for elements $C$ and $X$ in $\SL$.  We
  denote by $R^{\tau}(C_n)$ the set of $\tau$-equivariant
  $\SL$-representations.
\end{definition}

The $\tau$-equivariance of $\rho$ means that the pull--back of
$\rho$ by $\tau$ is expressed as the conjugation of an
$\SL$-element.

\begin{lemma}\label{lemma:uniqueness_C}
  If $\rho$ is irreducible and $\tau$-equivariant, 
then a matrix $C$ satisfying the diagram~$(\ref{diagram:equiv_diagram})$
is uniquely determined up to sign.
\end{lemma}
\begin{proof}
  It follows from Schur's lemma that for $\rho$ and $\tau^* \rho$
 the scalar matrices in $\SL$ are only $\pm \I$. 
\end{proof}

Similarly, we can also define $\tau$-equivariant representations for
$\pi_1(\branchedcoverZHS{n})$.  We denote by
$R^{\tau}(\branchedcoverZHS{n})$ the set of $\tau$-equivariant
representations of $\pi_1(\branchedcoverZHS{n})$.  
We use the same
symbol $\tau$ for the induced homomorphism on
$\pi_1(\branchedcoverZHS{n})$ by the next lemma.
\begin{lemma}\label{lemma:tau_equiv_Sigma_n}
  The following diagram is commutative:
  \begin{equation}\label{diagram:C2_Sigma2}
    \xymatrix{
      \pi_1(C_n) \ar[r]^{\tau} \ar[d]_{j_*} & \pi_1(C_n)\ar[d]^{j_*}\\
      \pi_1(\branchedcoverZHS{n}) \ar[r] &\pi_1(\branchedcoverZHS{n}).
    }
  \end{equation}
\end{lemma}
\begin{proof}
  The fundamental group $\pi_1(\branchedcoverZHS{n})$ is expressed as
  $\pi_1(C_n) / \nclosure{\mu_n}.$ Since $\tau$ is
  automorphism and $\tau(\mu_n) = \mu_n$ in $\pi_1(C_n)$.
  Note that $\nclosure{\mu_n}$ is expressed as
  $\{(g_1^{-1} \mu_n^{\epsilon_1} g_1) \cdots
  (g_k^{-1} \mu_n^{\epsilon_k} g_k) \,|\, k \geq 0, g_i \in \pi_1(C_n),
  \epsilon_i = \pm 1\, (1 \leq i \leq k) \}$.
  We have that 
  $\tau(\nclosure{\mu_n}) = \nclosure{\mu_n}$.  
  Hence $\tau$ induces an automorphism on the quotient group 
  $\pi_1(C_n) / \nclosure{\mu_n}$. 
\end{proof}
Moreover we can identify the set 
$R^{\tau}(\branchedcoverZHS{n})$
 with 
$\{\rho \in R^{\tau}(C_n)\,|\, \rho(\mu_n) = \I\}$
 by the
following Lemma~\ref{lemma:intersection_tau}.
\begin{lemma}\label{lemma:intersection_tau}
  The subset $\pullbackj(R^{\tau}(\branchedcoverZHS{n}))$ coincides
  with $\{\rho \in R^{\tau}(C_n)\,|\, \rho(\mu_n) = \I\}.$
\end{lemma}
\begin{proof}
  This follows from the commutative
  diagram~$(\ref{diagram:C2_Sigma2})$ in
  Lemma~\ref{lemma:tau_equiv_Sigma_n}. 
\end{proof}
The conjugation on representation spaces has the following property.
\begin{lemma}\label{lemma:Ad_equiv}
  The conjugation for representations preserves $\tau$-equivariance.
\end{lemma}
\begin{proof}
  Let $\rho$ be a $\tau$-equivariant representation of $\pi_1(C_n)$.
  There exists an element $C$ in $\SL$ such that the following diagram
  is commutative:
  \[
  \xymatrix{
    \pi_1(C_n) \ar[r]^{\tau} \ar[d]_{\rho} & \pi_1(C_n) \ar[d]^{\rho}\\
    \SL \ar[r]_{Ad_C} & \SL.  }
  \]
  Then for any $P$ in $\SL$, the representation $P \rho P^{-1}$
  satisfies that
  \[
  \xymatrix{
    \pi_1(C_n) \ar[r]^{\tau} \ar[d]_{P\rho P^{-1}} & \pi_1(C_n) \ar[d]^{P \rho P^{-1}}\\
    \SL \ar[r]_{Ad_{PCP^{-1}}} & \SL.  }
  \]
  Hence $P \rho P^{-1}$ is $\tau$-equivariant. 
\end{proof}

We will prove that $R^{\tau}(\branchedcoverZHS{n})$ coincides
with the image of $\Phi$.  To prove this, we use the identification 
$R^{\tau}(\branchedcoverZHS{n})
  = \{\rho \in R^{\tau}(C_n) \,|\, \rho(\mu_n) = \I \}$
in Lemma~\ref{lemma:intersection_tau} and the following lemma.
\begin{lemma}\label{lemma:Ad_C^n_action}
  Suppose that $\rho \in R^{\tau}(\branchedcoverZHS{n})$ and $C$ is
  a matrix as in the commutative
  diagram~$(\ref{diagram:equiv_diagram})$.  The adjoint action
  $Ad_{C^n}$ acts trivially on $\im \rho$.
\end{lemma}
\begin{proof}
  The $\tau$-equvariance of $\rho$ gives us the equality that
  $\rho(\tau^i(\gamma)) = Ad_{C^i}(\rho(\gamma))$ for any $\gamma$
  in $\pi_1(C_n)$.  For the case of $i=n$, 
  since $\tau^n(\gamma) = \mu_n\gamma \mu_n^{-1}$ and $\rho(\mu_n)=\I$, 
  we have that
  $Ad_{C^n}(\rho(\gamma)) = \rho(\gamma)$ for any $\gamma$ in
  $\pi_1(C_n)$, i.e, $Ad_{C^n}$ acts trivially on $\im \rho$. 
\end{proof}

We next show the following property on $\pullbackp$ to study the image of $\Phi$.
\begin{lemma}\label{lemma:pullback_R(Sigma_K)}
  For every $\rho$ in $R(\knotexterior)$, the image of 
  $\pullbackp \co R(\knotexterior) \to R(C_n)$ is contained in $R^{\tau}(C_n)$.
\end{lemma}
\begin{proof}
  Let $\gamma$ be an element in $\pi_1(C_n)$.  By the following direct
  calculation, it follows that $\pullbackp \rho$ is
  $\tau$-equivariant:
  \begin{align*}
    \pullbackp \rho (\tau(\gamma))
    &= \rho (\homPiICnEk(\tau(\gamma))) \\
    &= \rho (\mu \cdot \homPiICnEk(\gamma) \cdot \mu^{-1}) \\
    &= Ad_{\rho(\mu)} (\pullbackp \rho(\gamma)).
  \end{align*} 
\end{proof}
In fact, the map $p^*$ gives the following equality as sets:
\begin{proposition}\label{prop:image_Phi}
  $ \im \Phi = R^{\tau}(\branchedcoverZHS{n}). $
\end{proposition}
\begin{proof}
  For any representation $\rho$ in 
  $D(\Phi) \cap R_{\xi(k/n)}(\knotexterior)$ and 
  each $\gamma$ in $\pi_1(C_n)$, 
  the matrix $\Phi(\rho)(\tau(\gamma))$ is given by
  \begin{align*}
    \Phi(\rho)(\tau(\gamma))
    &= ( (\zeta_{2n})^{p_*[\tau(\gamma)]} )^k \cdot \pullbackp \rho(\tau(\gamma)) \\
    &= ( (\zeta_{2n})^{p_*[\gamma]} )^k \cdot Ad_{\rho(\mu)}(\pullbackp \rho(\gamma))
       \quad {(\rm by\, Lemma~\ref{lemma:pullback_R(Sigma_K)})}\\
    &= Ad_{\rho(\mu)}(\Phi(\rho)(\gamma)).
  \end{align*}
  Hence the $\SL$-representation $\Phi(\rho)$ is
  $\tau$-equivalent. 
  It is left to show the
  inclusion $\im \Phi \supset R^{\tau}(\branchedcoverZHS{n})$. 
  We take $\rho$ in
  $R^{\tau}(\branchedcoverZHS{n})$ and let $C$ be a matrix as in the
  commutative diagram $(\ref{diagram:equiv_diagram})$.  We use a
  presentation of $\pi_1(C_n)$ induced from a Lin presentation of
  $\pi_1(\knotexterior)$ as in
  Lemma~\ref{lemma:presentation_for_C_n}.
  The relations of $\pi_1(C_n)$ give the equalities:
  \[
  \rho(\mu_n \widetilde{\alpha}^{(0)}_i \mu_n^{-1}) =
  \rho(\widetilde{\beta}_i^{(n-1)}).
  \]
  The left hand side turns into $\rho(\widetilde{\alpha}^{(0)}_i)$.
  Since the element $\widetilde{\beta}_i^{(n-1)}$ can be written as
  $\tau^{n-1}(\widetilde{\beta}_i^{(0)})$, the right hand
  side $\rho(\widetilde{\beta}_i^{(n-1)})$ turns into
  $Ad_{C^{n-1}}(\rho(\widetilde{\beta}_i^{(0)}))$.  By
  Lemma~\ref{lemma:Ad_C^n_action}, $Ad_{C^n}$ acts trivially on $\im
  \rho$.  Hence we have the equalities
  \[
  C \rho(\widetilde{\alpha}^{(0)}_i) C^{-1}=
  \rho(\widetilde{\beta}_i^{(0)}).
  \]
  To make a representation $\rho_0$ of $\pi_1(\knotexterior)$ 
  satisfying $\Phi(\rho_0)=\rho$, we set
  $\rho_0 (\mu) = C$ and $\rho_0(x_i) = \rho(\widetilde{x}_i)$.  The
  relations 
  $\rho_0(\mu \alpha_i \mu^{-1}) = C \rho(\widetilde{\alpha}^{(0)}_i) C^{-1}$ 
  and 
  $\rho_0 (\beta_i) = \rho(\widetilde{\beta}_i^{(0)})$ show that
  \[\rho_0(\mu \alpha_i \mu^{-1}) = \rho_0(\beta_i)\] for all $i$.  
  Hence $\rho_0$ is an $\SL$-representation of $\pi_1(\knotexterior)$.  
  We must show that $\rho_0$ is contained in $D(\Phi)$.  
  By the following Lemma~\ref{lemma:on_value_C}, 
  there exists an integer $k$ such that
  $\rho_0$ is contained in $R_{\xi(k/n)}(\knotexterior)$. 
\end{proof}

\begin{lemma}\label{lemma:on_value_C}
  If $\rho$ is a nontrivial representation in
  $R^{\tau}(\branchedcoverZHS{n})$ and $C$ is an $\SL$ element as in
  the diagram~$(\ref{diagram:equiv_diagram})$, then $C^n = \pm \I$.
\end{lemma}
\begin{proof}
  We regard $\rho$ as an element in $R^{\tau}(C_n)$ satisfying
  $\rho(\mu_n)=\I$.  By Lemma~\ref{lemma:Ad_C^n_action}, the adjoint
  action $Ad_{C^n}$ acts trivially on $\im \rho$, i.e., $C^n$ commutes
  with all elements in $\im \rho \subset \SL$.  We suppose that $C^n$
  is not contained in the center $\{\pm \I\}$ in $\SL$.  This says
  that $C$ and $\im \rho$ are contained in the same maximal abelian
  subgroup in $\SL$.  Since $\rho$ is $\tau$-equivariant and $C$
  commutes with all elements of $\im \rho$, we have that
  $\rho(\tau(\gamma)) = Ad_C(\rho(\gamma)) = \rho(\gamma)$ for any
  $\gamma$ in $\pi_1(C_n)$.  Then $\rho$ turns into the pull--back of
  abelian representation of $\pi_1(\knotexterior)$  given by the
  following corresponding:
  \[
  \mu \mapsto C, \quad x_i \mapsto \rho(\widetilde{x}_i).
  \]
  for Lin presentations of $\pi_1(\knotexterior)$ and $\pi_1(C_n)$.
  However every abelian representation of $\pi_1(\knotexterior)$
  sends $x_i$ to $\I$ since $x_i$ is an element in the commutator
  subgroup.  Thus $\rho$ is trivial but this contradicts to our
  assumption that $\rho$ is nontrivial. 
\end{proof}

Now, Proposition \ref{prop:image_Phi} can be naturally extended to
the set of characters as follows:
\begin{proposition}
  The image of $\hPhi$ coincides with the fixed point set 
  $X(\branchedcoverZHS{n})^{\tau}$ of the
  action $\tau^*$ on $X(\branchedcoverZHS{n})$.
\end{proposition}
\begin{proof}
  It follows by definition that the image $\im \hPhi =
  t(R^{\tau}(\branchedcoverZHS{n}))$ is contained in
  $X(\branchedcoverZHS{n})^{\tau}$.  We need to prove that every fixed point
  of $\tau^*$ in $X(\branchedcoverZHS{n})$ is given by the character
  of a $\tau$-equivariant representation.

  First we consider the character 
  $\chi_\rho \in X(\branchedcoverZHS{n})^{\tau}$ associated to 
  an irreducible representation $\rho$ such that it is fixed by $\tau^*$, 
  i.e.,
  $\tau^* (\chi_\rho) = \chi_{\tau^* \rho} = \chi_{\rho}$.  Since
  $\rho$ is irreducible, the representation $\rho$ is conjugate to
  $\tau^* \rho$.  This says that $\rho$ is $\tau$--equivariant.

  Next we consider the case that 
  $\chi_\rho \in X(\branchedcoverZHS{n})^{\tau}$ is 
  a reducible character such that $\tau^*(\chi_\rho) = \chi_\rho$.  
  For every reducible representation there exists an abelian representation of
  $\pi_1(\branchedcoverZHS{n})$ such that it has the same character as
  that of the reducible one and the image is contained in the maximal
  abelian subgroup ${\rm Hyp}$.  So we can assume without loss of
  generality that $\im \rho$ is contained in ${\rm Hyp}$.  By using a
  presentation as in Lemma~\ref{lemma:presentation_for_Sigma_n} and
  $\tau^*(\chi_\rho) = \chi_\rho$, we have the equality $\trace
  \rho(\tau^{j}(\wtilde{x}_i)) = \trace \rho({\wtilde{x}_i})$.
  Hence $\rho(\tau^j(\wtilde{x}_i))$ is either $\rho(\wtilde{x}_i)$
  or $\rho(\wtilde{x}_i)^{-1}$.  Set $\rho(\wtilde{x}_i)$ and
  $\rho(\tau^j(\wtilde{x}_i))$ as
  \[
  \rho(\wtilde{x}_i) = \left(
    \begin{array}{cc}
      r_i e^{\sqrt{-1}\theta_i} & 0 \\
      0 & r_i^{-1} e^{-\sqrt{-1}\theta_i}
    \end{array}
  \right), \quad \rho(\tau^j(\wtilde{x}_i)) =
  \rho(\wtilde{x}_i)^{\epsilon_j} \quad (\epsilon_j = \pm 1).
  \]
  Since $\rho$ factors through $H_1(\branchedcoverZHS{n};\Z)$, the
  relation $-U\bm{x_j} + V\bm{x}_{j+1}$ in
  $H_1(\branchedcoverZHS{n};\Z)$ gives us the following relation:
  \begin{equation}\label{eqn:relation_j_j+1}
    (-\epsilon_j U + \epsilon_{j+1} V)
      \begin{pmatrix}
        \log r_1 \\
        \vdots\\
        \log r_{2g}
      \end{pmatrix}
     = \bm{0},\
    (-\epsilon_j U + \epsilon_{j+1} V)
      \begin{pmatrix}
        \theta_1 \\
        \vdots\\
        \theta_{2g}
      \end{pmatrix}
     \equiv \bm{0}\,\ ({\rm mod}\,2\pi\Z).
  \end{equation}
  If there exists some $j$ such that $\epsilon_{j+1} = \epsilon_j$,
  then it follows by the relation~$(\ref{eqn:relation_j_j+1})$ and
  $\det (-U+V) = \Delta_K(1) = \pm 1$ that $\rho(\wtilde{x}_i)=\I$ for
  all $i$, i.e, $\rho$ is trivial.  
  If the equality $\epsilon_{j+1} = - \epsilon_j$ holds for all $j$, 
  then we have $\tau^* \rho = C \rho C^{-1}$ 
  where 
  $C = \left(\begin{array}{cc}
      0 & 1 \\
      -1 & 0
    \end{array}
  \right)$.  Thus $\rho$ is $\tau$-equivariant in the both case.
  Therefore $\im \hPhi$ coincides with the fixed point set
  $X(\branchedcoverZHS{n})^{\tau}$. 
\end{proof}

In the case of $\SU$-representations, such a correspondence like
$\Phi$ has been considered by Collin and Saveliev
in~\cite{CollinSaveliev01}. They deal with irreducible
$\SU$-representations mainly.  Here we also draw attention to reducible
representations of $\pi_1(\branchedcoverZHS{n})$.  Lin presentations
lead us to the property of $\Phi$ for reducible representations.

\begin{proposition}\label{prop:reducible_abelian}
  In the image of $\Phi$, all reducible representations are abelian,
  i.e.,
  \begin{equation}\label{eqn:reducible_abelian}
    \im \Phi \cap R^{\mathrm{red}}(\branchedcoverZHS{n}) = 
    \im \Phi \cap R^{\mathrm{ab}}(\branchedcoverZHS{n}).
  \end{equation}
\end{proposition}
\begin{proof}
  It is obvious that any abelian representation is
  a reducible one.  We must show
  that the left hand side of Eq.~$(\ref{eqn:reducible_abelian})$ is
  contained in the right hand side.

  By Proposition~\ref{prop:image_Phi}, $\im
  \Phi \cap R^{\mathrm{red}}(\branchedcoverZHS{n})$ coincides with
  $R^{\tau}(\branchedcoverZHS{n}) \cap R^{\mathrm{red}}(C_n)$ in
  $R(C_n)$.  So we can assume that $\rho$ is an element in
  $R^{\mathrm{red}}(C_n)$ such that $\rho(\mu_n) = \I$ and $C$ is an
  $\SL$-element satisfying $\tau^*\rho = Ad_{C}(\rho)$.  We use the
  presentation of $\pi_1(C_n)$ as in
  Lemma~\ref{lemma:presentation_for_C_n}.  Since $\rho$ is reducible,
  by taking conjugation, we can assume that
  \[
  \rho(\tau^j(\wtilde{x}_i)) =\left(
    \begin{array}{cc}
      \alpha_{i, j} & \beta_{i, j} \\
      0 & \alpha_{i, j}^{-1}
    \end{array}
  \right)
  \]
  for all $i$ and $j$.  If $\alpha_{i, j}$ is $\pm 1$ for all $i$ and
  $j$, then $\im \rho$ is contained in the maximal abelian subgroup
  ${\rm Para}$.  Hence $\rho$ is abelian.  

  We consider the case that there exists a pair $(i, j)$ with
  $\alpha_{i, j} \not = \pm 1$. 
  We can assume that $\rho(\tau^j(\wtilde{x}_i))$ is a diagonal matrix
  by taking conjugation of an upper triangular matrix.  Since $\rho$
  is $\tau$-equivariant and reducible,
  $\rho(\tau^{j+1}(\widetilde{x}_i)) =
  C\rho(\tau^{j}(\widetilde{x}_i))C^{-1}$ is also an upper triangular
  matrix.  Writing $C=\begin{pmatrix}
    a & b \\
    c & d
     \end{pmatrix}$, 
  we have that $cd = 0$.  If $c$ were zero,
  then we would have the following contradiction.  When we set
  $\rho_0(\mu) = C$ and $\rho_0(x_i) = \rho(\wtilde{x}_i)$, the
  correspondence $\rho_0$ defines an $\SL$-representation of
  $\pi_1(\knotexterior)$.  Moreover since the image of $\rho_0$
  consists of upper triangular matrices, the representation $\rho_0$
  is reducible.  Hence for every commutator $\gamma$ in
  $\pi_1(\knotexterior)$, the trace of $\rho_0(\gamma)$ is $\pm
  2$. Applying this to the commutator $x_i$, we have that $\trace
  \rho_0(x_i) = \alpha_{i, j} + \alpha_{i, j}^{-1} = \pm 2$. This is a
  contradiction to our assumption that $\alpha_{i, j} \not = \pm 1$.

  Now we have  $d=0$.  Furthermore since
  $\rho(\tau^{j+h}(\wtilde{x}_i)) = C^h\rho(\tau^j(\wtilde{x}_i))C^{-h}$ 
  must be upper triangular, it follows that the matrix $C$ turns into
  $\begin{pmatrix}
      0 & b \\
      -b^{-1} & 0
    \end{pmatrix}$.
  By reducibility of $\rho$, the matrix $C \rho(\tau^l(\wtilde{x}_k)) C^{-1}$
  is also an upper triangular matrix. Therefore every $\beta_{k, l}$ must be zero. 
  Then the image of $\rho$ is contained in the maximal abelian subgroup ${\rm Hyp}$,
  i.e., $\rho$ is also abelian in this case. This completes the
  proof. 
\end{proof}

\begin{proposition}\label{prop:preimage_red}
  The preimage $\Phi^{-1}(R^{\mathrm{red}}(\branchedcoverZHS{n}))$
 is expressed as
  \[
  \Phi^{-1}(R^{\mathrm{red}}(\branchedcoverZHS{n})) =
  R^{\mathrm{meta}}(\knotexterior) \cap D(\Phi).
  \]
\end{proposition}
\begin{proof}
  By Proposition \ref{prop:reducible_abelian}, it is enough to show
  that the preimage
  $\Phi^{-1}(R^{\mathrm{ab}}(\branchedcoverZHS{n}))$ coincides with
  $R^{\mathrm{meta}}(\knotexterior) \cap D(\Phi)$.  Let $\rho$ be
  an element in $\Phi^{-1}(R^{\mathrm{ab}}(\branchedcoverZHS{n}))$.
  By taking conjugation, we can assume that the image $\Phi(\rho)$
  is contained in either the maximal abelian subgroups ${\rm Hyp}$ or
  ${\rm Para}$.  We choose a Lin presentation of
  $\pi_1(\knotexterior)$ as in Lemma~\ref{lemma:Lin_presentation}.
  In the case that $\im \Phi(\rho)$ is contained in ${\rm Hyp}$, for
  all $i$ the matrix $\rho(x_i)$ is expressed as
  \[
  \rho(x_i) = \left(
    \begin{array}{cc}
      \alpha_i & 0 \\
      0 & \alpha_i^{-1}
    \end{array}
  \right).
  \]
  If all $\alpha_i$ are $\pm 1$, then $\rho$ is abelian.  We consider
  the case that there exist some $\alpha_i$ such that 
  $\alpha_i \not = \pm 1$.  Since 
  $\Phi(\rho)(\tau(\wtilde{x}_i)) = \pm \rho(\mu)\rho(x_i)\rho(\mu)^{-1}$ 
  and $\Phi(\rho)(\tau(\wtilde{x}_i))$ is also upper triangular, 
  the matrix $\rho(\mu)$ is expressed as either
  \[
  \left(
    \begin{array}{cc}
      a & 0 \\
      0 & a^{-1}
    \end{array}
  \right) \quad {\text or} \quad 
  \left(
    \begin{array}{cc}
      0 & b \\
      -b^{-1} & 0
    \end{array}
  \right).
  \]
  In both cases, the representation $\rho$ turns into metabelian
  (refer to Eq.~$(\ref{eqn:explicit_binary_dihedral}$)).

  In the case that $\im \Phi(\rho)$ is contained in ${\rm Para}$,
  for all $i$, $\rho(x_i)$ is expressed as
  \[
  \rho(x_i) = \pm \left(
    \begin{array}{cc}
      1 & \beta_i \\
      0 & 1
    \end{array}
  \right).
  \]
  If all $\beta_i$ are zero, then $\rho$ is abelian.  We consider the
  case that there exists some non--zero $\beta_i$.  Since
  $\Phi(\rho)(\tau(\wtilde{x}_i)) = \pm \rho(\mu)\rho(x_i)\rho(\mu)$ 
  is also a parabolic element, the matrix
  $\rho(\mu)$ must be upper triangular.  Therefore $\rho$ is
  reducible, in particular, it is metabelian.

  On the other hand, by the proof
  of Proposition~\ref{prop:meta_conj_dihedral} and 
  Remark~\ref{remark:metabelian}, every metabelian
  representation is conjugate to either
  \[
  \rho(x_i) = \left(
    \begin{array}{cc}
      \alpha_i & 0 \\
      0 & \alpha_i^{-1}
    \end{array}
  \right),\quad \rho(\mu) = \left(
    \begin{array}{cc}
      0 & 1 \\
      -1 & 0
    \end{array}
  \right)
  \]
  or
  \[
  \rho(x_i) = \left(
    \begin{array}{cc}
      1 & \omega_i \\
      0 & 1
    \end{array}
  \right),\quad \rho(\mu) = \left(
    \begin{array}{cc}
      \lambda & \beta \\
      0 & \lambda^{-1}
    \end{array}
  \right).
  \]
  It follows from the definition of $\Phi$ that all metabelian
  representations in $D(\Phi)$ are sent to abelian ones of
  $\pi_1(\branchedcoverZHS{n})$. 
\end{proof}

\section{On the geometry of the slice $S_0(\knotexterior)$}
\label{section:slice_S0}
In the subsequent section, we focus on knots in $S^3$ mainly 
to see the interesting structure of the slice $S_0(\knotexterior)$ 
through the map $\hPhi$. 
As described below, when $\hPhi$ is surjective, it describes the slice 
$S_0(\knotexterior)$ as the $2$--fold branched cover 
over the $\SL$-character variety of $B_2$
with branched locus given by the abelian characters, whose preimage is
precisely the set of metabelian characters.
Note that the results in this section
can be naturally extended to the case of an integral homology
$3$--sphere with some modifications.

First we describe the representation spaces $R_0(\knotexterior)$ and
$R^{\tau}(\branchedcoverZHS{2})$ in detail.

\begin{lemma}\label{lemma:red_R0}
  If $\rho \in R_0(\knotexterior)$, then $\rho(\mu)^2 = -\I$.
  Therefore the domain of $\PhiII$ is given by the union
  $R_0(\knotexterior) \cup R(S^3) \cup \iota(R(S^3)) =
  R_0(\knotexterior) \cup \{{\rm trivial\, rep.}\} \cup \{\iota({\rm trivial\, rep.})\}$.
  Moreover the image of $\Phi$ coincides with that of the
  restriction on $R_0(\knotexterior)$.
\end{lemma}
\begin{proof}
  This follows from the Cayley--Hamilton identity and Remark
  \ref{remark:domain}.  The equality $\im \Phi = \im
  \Phi|_{R_0(\knotexterior)}$ can be deduced from the fact that
  $R(S^3)$ and $\iota(R(S^3))$ consist of only the trivial
  representation and its image by $\iota$, respectively.  It follows from
  Lemma~\ref{lemma:image_abelian} that they are sent to the trivial
  representation of $\pi_1(\branchedcoverZHS{2})$ by $\PhiII$ and all
  abelian representations in $R_0(\knotexterior)$ are also sent to the
  trivial one of $\pi_1(\branchedcoverZHS{2})$.  This shows that $\im
  \Phi = \im \Phi|_{R_0(\knotexterior)}$. 
\end{proof}
By Lemma~\ref{lemma:red_R0}, in the case of knots in $S^3$, the domain
$D(\PhiII)$ is reduced to the slice $R_0(\knotexterior)$.  This makes
the properties on $\PhiII$ and $R_0(\knotexterior)$ much clearer.  For
example, we obtain the next proposition by
Propositions~\ref{prop:image_Phi}, \ref{prop:reducible_abelian}, \ref{prop:preimage_red} 
and Lemma~\ref{lemma:image_abelian}.
\begin{proposition}\label{prop:correspond_rep}
  The map $\PhiII$ from $R_0(\knotexterior)$ onto
  $R^{\tau}(\branchedcoverZHS{2})$ makes the following
  correspondence. For an element $\rho$ in $R_0(\knotexterior)$,
  \begin{enumerate}
  \item If $\rho$ is abelian, then $\PhiII(\rho)$ is the trivial
    representation of $\pi_1(\branchedcoverZHS{2})$;
  \item If $\rho$ is irreducible and metabelian, then $\PhiII(\rho)$
    is a non-trivial abelian representation in
    $R^{\tau}(\branchedcoverZHS{2})$; and
  \item If $\rho$ is non-metabelian, in particular irreducible, then
    $\PhiII(\rho)$ is an irreducible representation in
    $R^{\tau}(\branchedcoverZHS{2})$.
  \end{enumerate}
\end{proposition}

Recall that the map $\PhiII$ induces the map $\hPhi:S_0(\knotexterior)
\to X(\branchedcoverZHS{2})$ defined by
$\hPhi(\chi_{\rho}):=\chi_{\PhiII(\rho)}$ (see also
Eq.~$(\ref{eqn:def_hatPhi})$).  This gives us the following theorem,
which is one of the main results in this article.
\begin{theorem}\label{thm:geom_S0}
  The image of $S_0(\knotexterior)$ by $\hPhi$ coincides with the
  subset $X(\branchedcoverZHS{2})^{\tau} = \{\chi_\rho \in
  X(\branchedcoverZHS{2})\,|\, \rho \co \tau{\text -equivariant}\}$.
  The map $\hPhi \co S_0(\knotexterior) \to
  X(\branchedcoverZHS{2})^{\tau}$ is one-to-one correspondence on the
  fixed point set $S_0(\knotexterior)^{\hiota}$ and
  two-to-one correspondence on $S_0(\knotexterior) \setminus
  S_0(\knotexterior)^{\hiota}$.  Moreover two points in the preimage
  of a point in $X(\branchedcoverZHS{2})^{\tau}$ are interchanged 
  by the involution $\hiota$ on $S_0(\knotexterior)$.
\end{theorem}
The rest of this section is devoted to proving Theorem~\ref{thm:geom_S0}.
\begin{lemma}\label{lemma:Phi_and_iota}
  If $\rho$ is an element in $D(\PhiII)$, then the map $\PhiII$ sends
  both $\rho$ and $\iota(\rho)$ to the same $\SL$-representation of
  $\pi_1(\branchedcoverZHS{2})$, i.e.,
  $\PhiII(\rho)=\PhiII(\iota(\rho))$.  Moreover, this equality gives
  $\hPhi(\chi_\rho) = \hPhi(\chi_{\iota(\rho)})$.
\end{lemma}
\begin{proof}
  The map $\PhiII$ is constructed by the action $\iota$ on $R(C_2)$
  and the pull-back $p^*$.  For each element $\rho$ in
  $R_0(\knotexterior)$, the $\SL$-representation $\PhiII (\rho)$ is
  given by $\iota(p^* \rho)$ and $\PhiII (\iota(\rho))$ by $\iota(p^*
  \iota(\rho))$.  Moreover it follows from $p_* (H_1(C_2;\Z)) \simeq
  2\Z$ in $H_1(\knotexterior;\Z) \simeq \Z$ that for any $g$ in
  $\pi_1(C_2)$
  \[p^* \iota(\rho)(g) 
  =  (-1)^{p_*[g]} p^*\rho (g)
  = p^*\rho (g).
  \]
  Hence $\PhiII(\iota(\rho))$ coincides with $\PhiII(\rho)$.  This
  equality also holds for the induced map $\hPhi$. 
\end{proof}
By Proposition \ref{prop:fixed_point}, the fixed point set
$S_0(\knotexterior)^{\hiota}$ in $S_0(\knotexterior)$ consists
entirely of metabelian characters, which are in $R_0(\knotexterior)$.
\begin{proposition}\label{prop:one_to_one}
  The restriction of $\hPhi$ on $S_0(\knotexterior)^{\hiota}$ gives a
  bijection between $S_0(\knotexterior)^{\hiota}$ and all abelian
  characters of $\pi_1(\branchedcoverZHS{2})$.
\end{proposition}
\begin{proof}
  By Proposition~\ref{prop:fixed_point} and
  Lemma~\ref{lemma:fixed_pt_in_S0}, we focus on metabelian characters.
  By Propositions~\ref{prop:reducible_abelian} and
  \ref{prop:preimage_red}, the map $\hPhi$ sends the metabelian
  characters of $\pi_1(\knotexterior)$ onto the abelian characters of
  $\pi_1(\branchedcoverZHS{2})$.  By
  Propositions~\ref{prop:meta_conj_dihedral} and
  \ref{prop:fixed_point_set}, 
  the number of metabelian characters
  is equal to $(|\Delta_K(-1)|-1)/2 + 1$ 
  where
  $\Delta_K(t)$ is the Alexander polynomial of $K$.  On the other
  hand, we see that 
  the number of abelian characters of
  $\pi_1(\branchedcoverZHS{2})$ is also equal to 
  $(|\Delta_K(-1)|-1)/2 + 1$ 
  as follows.

  Every abelian representation factors through the abelianization
  $H_1(\branchedcoverZHS{2};\Z)$.  Since the order of
  $H_1(\branchedcoverZHS{2};\Z)$ is finite (an odd integer
  $|\Delta_K(-1)|$), $H_1(\branchedcoverZHS{2};\Z)$ is decomposed into
  the direct sum of some finite cyclic groups.  The character of an
  abelian representation is determined by the traces for generators of
  these cyclic groups.  By conjugation, the images of generators of
  cyclic groups are given by the diagonal matrices whose diagonal
  components are roots of unity (because the order of each cyclic
  group is finite).  Combining these facts, we can show that the
  number of abelian characters is given by
  $(|\Delta_K(-1)|-1)/2 + 1$.  Therefore the restriction of $\hPhi$
  on $S_0(\knotexterior)^{\hiota}$ gives a one--to--one
  correspondence. 
\end{proof}

\begin{proposition}\label{prop:two_to_one}
  The restriction of $\hPhi$ to $S_0(\knotexterior) \setminus
  S_0(\knotexterior)^{\hiota}$ gives a two-to-one correspondence.
\end{proposition}

\begin{proof}
  It is sufficient to show that if $\rho$ and $\rho'$ are two
  non-metabelian representations of $\pi_1(\knotexterior)$ satisfying
  $\hPhi(\chi_\rho) = \hPhi(\chi_{\rho'})$, then $\chi_{\rho'}$
  coincides with either $\chi_\rho$ or $\hiota(\chi_{\rho})$.  By
  Proposition~\ref{prop:correspond_rep}, the representations
  $\PhiII(\rho)$ and $\PhiII(\rho')$ are irreducible.  Since
  $\PhiII(\rho)$ and $\PhiII(\rho')$ have the same characters, by
  \cite[Proposition 1.5.2]{CS:1983}, these two representations are
  conjugate, i.e., there exists an $\SL$-element $P$ such that
  \begin{equation}\label{eqn:conjugate_Phi}
    \PhiII(\rho') = P \PhiII(\rho) P^{-1}.
  \end{equation}
  We use the presentations of $\pi_1(\knotexterior)$ and
  $\pi_1(\branchedcoverZHS{2})$ as in
  Lemma~\ref{lemma:presentation_for_Sigma_n}.  By the construction of
  $\PhiII$, the matrix $\PhiII(\rho)(\wtilde{x}_i)$ is expressed as
  \begin{equation}\label{eqn:Phi_x_i}
    \PhiII(\rho)(\wtilde{x}_i) = \rho(x_i).
  \end{equation}
  Combining Eq.~$(\ref{eqn:Phi_x_i})$ with
  Eq.~$(\ref{eqn:conjugate_Phi})$, we have the equality
  \begin{equation}\label{eqn:rho_for_commutator}
    \rho'(x_i) = P \rho(x_i) P^{-1}.
  \end{equation}
  Similarly, the matrix $\PhiII(\rho)(\tau(\wtilde{x}_i))$ is
  expressed as
  \begin{equation}\label{eqn:Phi_tau_x_i}
    \PhiII(\rho)(\tau(\wtilde{x}_i)) = \rho(\mu)\rho(x_i)\rho(\mu)^{-1} 
    = Ad_{\rho(\mu)}(\PhiII(\rho)(\wtilde{x}_i)).
  \end{equation}
  Combining Eq.~$(\ref{eqn:Phi_tau_x_i})$ with
  Eq.~$(\ref{eqn:conjugate_Phi})$, we have the equality
  \[
  \PhiII(\rho')(\tau(\wtilde{x}_i)) = Ad_{P \rho(\mu)
    P^{-1}}(\PhiII(\rho')(\wtilde{x}_i)).
  \]
  The matrix $\PhiII(\rho)(\tau(\wtilde{x}_i))$ is also expressed as
  $Ad_{\rho'(\mu)}(\PhiII(\rho')(\wtilde{x}_i)).$ By
  Lemma~\ref{lemma:uniqueness_C}, we obtain the equality
  \begin{equation}\label{eqn:rho_for_meridian}
    \rho'(\mu) = \pm P\rho(\mu)P^{-1}.
  \end{equation}
  From the relations that $\iota(\rho)(x_i) = \rho(x_i)$ and
  $\iota(\rho)(\mu) = - \rho(\mu)$ and
  Eqs.~$(\ref{eqn:rho_for_commutator})$ and
  $(\ref{eqn:rho_for_meridian})$, we can conclude that $\rho'$ is
  conjugate to either $\rho$ or $\iota(\rho)$ by $P$. 
\end{proof}

\begin{proof}[Proof of Theorem \ref{thm:geom_S0}]
  Combining Lemma~\ref{lemma:Phi_and_iota},
  Propositions~\ref{prop:one_to_one} and \ref{prop:two_to_one}, we
  obtain Theorem~\ref{thm:geom_S0}. 
\end{proof}

Moreover $\tau$-equivariant representations subspace
$R^{\tau}(\branchedcoverZHS{2})$ have the following distinct property
from the cases other than $n=2$ as mentioned at the beginning of
Section~\ref{section:map_rep_spaces} 
(compare to Proposition~\ref{prop:abelian_but_not_equivalent}).

\begin{lemma}\label{lemma:tau_equiv_contains_abelian}
  The set $R^{\tau}(\branchedcoverZHS{2})$ contains
  $R^{\mathrm{ab}}(\branchedcoverZHS{2})$.
\end{lemma}
\begin{proof}
  Let $\rho$ be an abelian $\SL$-representation of
  $\pi_1(\branchedcoverZHS{2})$.  Since this representation factors
  through $H_1(\branchedcoverZHS{2};\Z)$ which is a finite abelian
  group with the order $|\Delta_K(-1)|$.  By taking conjugation, the
  image $\im \rho$ is contained in the maximal abelian subgroup ${\rm
    Hyp}$.  So we can set $\rho(\widetilde{x}_i)$ as
  \[
  \rho(\wtilde{x}_i) = \left(
    \begin{array}{cc}
      \alpha_i & 0 \\
      0 & \alpha_i^{-1}
    \end{array}
  \right).
  \]
  By Lemma~\ref{lemma:presentation_for_Sigma_n}, we have
  $[\wtilde{x}_i] + [\tau(\wtilde{x}_i)] = \bm{0}$ in
  $H_1(\branchedcoverZHS{2};\Z)$ for all $i$.  Hence
  $\rho(\tau(\wtilde{x}_i))$ is given by $\rho(\wtilde{x}_i)^{-1}$.

  If we set 
  $C = \left(
  \begin{array}{cc}
    0 & 1 \\
    -1 & 0
  \end{array}
  \right)$, then the diagram~$(\ref{diagram:equiv_diagram})$ becomes
  commutative.  Therefore $\rho$ is $\tau$-equivariant. 
\end{proof}

\begin{remark}
  Lemma~\ref{lemma:tau_equiv_contains_abelian} does not hold for
  $3$--fold branched covers of $S^3$.  As shown below, for the trefoil
  knot in $S^3$ there exists an abelian representation of
  $\pi_1(\Sigma_3)$ which is not $\tau$-equivariant.
\end{remark}

\begin{proposition}\label{prop:abelian_but_not_equivalent}
  Let $K$ be the left hand trefoil knot in $S^3$.  The set
  $R^{\mathrm{ab}}(\Sigma_3)$ is not contained in
  $R^{\tau}(\Sigma_3)$.
\end{proposition}

To show Proposition~\ref{prop:abelian_but_not_equivalent}, 
we consider regular Seifert surface $S$ of the left hand trefoil 
depicted in Figure~\ref{fig:Seifert_trefoil}. 
We fix a spine of $S$ as in Figure~\ref{fig:Seifert_trefoil}. 
Then the Seifert matrix is given by the following matrix 
\[
Q=
\left(
  \begin{array}{cc}
    1 & 0 \\
    -1 & 1
  \end{array}
\right).
\]
Now we consider the Heegaard splitting 
$\overline{S^3 \setminus N(S)} \cup N(S)$ associated to $S$
as in Figure~\ref{fig:Heegaard_trefoil}. Each attaching circle of
$2$-handles % two--handles
intersects with the boundaries of meridian disks of
$1$-handles % one--handles
at only one point in this Heegaard splitting. 
In this setting, the boundary operator from $C^{\hbox{\scriptsize Morse}}_2(S^3;\Z)$ 
to $C^{\hbox{\scriptsize Morse}}_1(S^3;\Z)$ is expressed as the identity matrix 
(see Remark~\ref{remark:matrix_T_identity}). 
The relations in the Lin presentation for $S$ are given by
\[
\alpha_1 = x_1x_2^{-1}, \quad \alpha_2 = x_2, \quad \beta_1 = x_1,
\quad \beta_2 = x_1^{-1}x_2.
\]
Hence, by Lemma~\ref{lemma:presentation_for_Sigma_n}, a presentation
of $\pi_1(\Sigma_3)$ is given by
\[
\langle \wtilde{x}_1, \wtilde{x}_2, \tau(\wtilde{x}_1),
\tau(\wtilde{x}_2), \tau^2(\wtilde{x}_1), \tau^2(\wtilde{x}_2),
\,| \, \wtilde{\alpha}_1^{(k)} = \wtilde{\beta}_1^{(k-1)},
\wtilde{\alpha}_2^{(k)} = \wtilde{\beta}_2^{(k-1)}\, (k\,{\rm mod}\,
3) \rangle.
\]
This presentation is reduced as 
\[
\langle \wtilde{x}_1, \wtilde{x}_2 \,|\, \wtilde{x}_1
\wtilde{x}_2^{-1} = \wtilde{x}_2 \wtilde{x}_1, \wtilde{x}_2
\wtilde{x}_1 = \wtilde{x}_1^{-1} \wtilde{x}_2 \rangle.
\]
About another method to obtain this presentation, 
for example see~\cite[Chapter 10]{Rolfsen}.

\begin{figure}[!ht]
  \begin{center}
    \includegraphics[scale=.5]{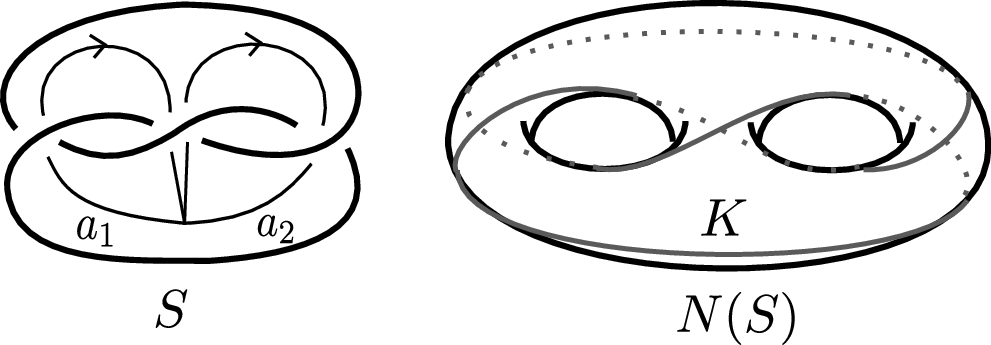}
  \end{center}
  \caption{The regular Seifert surface of the trefoil knot}\label{fig:Seifert_trefoil}
\end{figure}

\begin{figure}[!ht]
  \begin{center}
    \includegraphics[scale=.4]{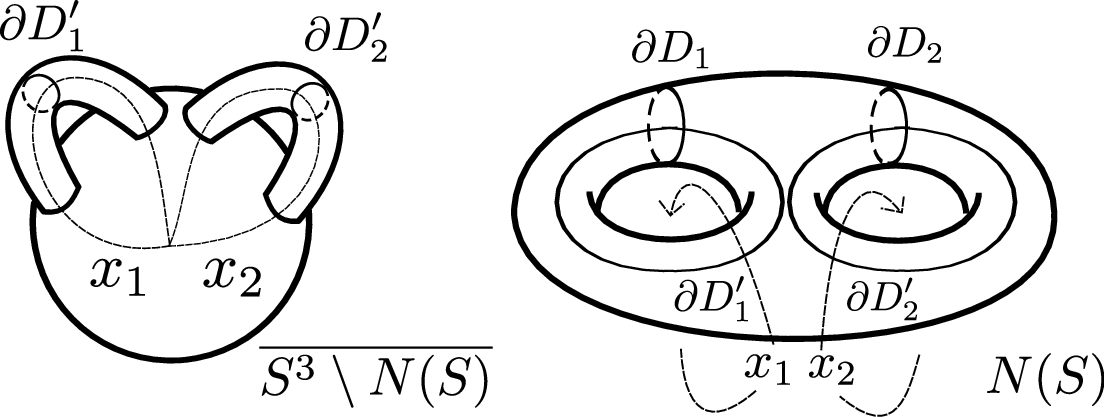}
  \end{center}
  \caption{Heegaard splitting of $S^3$ associated to $S$}
  \label{fig:Heegaard_trefoil}
\end{figure}

\begin{proof}[Proof of Proposition~\ref{prop:abelian_but_not_equivalent}]
  By Lemma~\ref{lemma:presentation_for_Sigma_n}, the homology group
  $H_1(\Sigma_3;\Z)$ is presented as
  \[
  \Z^6 \xrightarrow{A} \Z^6 \to H_1(\Sigma_3;\Z),
  \]
  where
  \[
  A = \left(
    \begin{array}{ccc}
      \trans{Q} & -Q & 0 \\
      0 & \trans{Q} & -Q \\
      -Q & 0 & \trans{Q} \\
    \end{array}
  \right).
  \]
  If we set
  \[
  \left\{
    \begin{array}{lll}
      \rho(\wtilde{x}_1) = \I, & \rho(\tau(\wtilde{x}_1)) 
      = -\I, & \rho(\tau^2(\wtilde{x}_1)) = -\I, \\
      \rho(\wtilde{x}_2) = -\I, & \rho(\tau(\wtilde{x}_2)) 
      = \I, & \rho(\tau^2(\wtilde{x}_2)) = -\I, \\
    \end{array}
  \right.
  \]
  then this correspondence defines an abelian representation of
  $\pi_1(\Sigma_3)$.  Since the trace of $\rho(\wtilde{x}_1)$ is not
  equal to that of $\rho(\tau(\wtilde{x}_1))$, there does not exist
  any $\SL$-element such that the
  diagram~$(\ref{diagram:equiv_diagram})$ commutes. 
\end{proof}

\section{Applications}
\label{section:application}
In the previous section, we have investigated the correspondence
$\hPhi$ between $S_0(\knotexterior)$ and
$X(\branchedcoverZHS{2})^{\tau}$.  In particular,
Lemma~\ref{lemma:tau_equiv_contains_abelian} shows that
$R^{\mathrm{ab}}(\branchedcoverZHS{2}) \subset
R^{\tau}(\branchedcoverZHS{2})$.  In this section, we look into
surjectivity of $\hPhi \co S_0(\knotexterior) \to
X(\branchedcoverZHS{2})$ for two--bridge knots and pretzel knots of
type $(p,q,r)$.  Namely, we show that
$R^{\tau}(\branchedcoverZHS{2})$ contains all $\SL$-representations
$R(\branchedcoverZHS{2})$.

As regards two--bridge knots, it is well-known that the two--fold
branched cover $\branchedcoverZHS{2}$ along a two--bridge knot is a
$3$-dimensional lens space.  Since the fundamental group of a lens
space is cyclic, the $\SL$-representation space for a lens space
consists entirely of abelian representations.  By
Theorem~\ref{thm:geom_S0} and Proposition~\ref{prop:correspond_rep},
the slice $S_0(\knotexterior)$ consists of metabelian characters of
$\pi_1(\knotexterior)$. It is known that the order of
$H_1(\branchedcoverZHS{2};\Z)$ is given by $|\Delta_K(-1)|$.  Thus we
have proved the following Lemma, which follows originally from the
proof of Theorem $1.3$ in \cite{nagasato07:_finit_of_section_of_sl}.
\begin{lemma}\label{lemma:two_brige_S0}
  For a two--bridge knot $K$, the slice $S_0(\knotexterior)$ coincides
  with the fixed point set $S_0(\knotexterior)^{\hiota}$ and it can be
  identified with the set of $((|\Delta_K(-1)|-1)/2+1)$ conjugacy
  classes of the $\SL$-metabelian representations of
  $\pi_1(\knotexterior)$.
\end{lemma}
\begin{remark}
  We do not have to consider reducible and non--abelian
  representations in Lemma~\ref{lemma:two_brige_S0} since there exist
  no such representations in $R_0(\knotexterior)$.  This is due to
  Remark~\ref{remark:reducible_reps_in_S0}.
\end{remark}
Another way of stating Lemma~\ref{lemma:two_brige_S0} is to say that
if $\branchedcoverZHS{2}$ is a lens space, then the slice
$S_0(\knotexterior)$ coincides with the fixed point set
$S_0(\knotexterior)^{\hiota}$.  This means that the difference
$S_0(\knotexterior) \setminus S_0(\knotexterior)^{\hiota}$ is an
obstruction for $\branchedcoverZHS{2}$ to be a lens space 
and thus for $K$ to be a two--bridge knot.
\begin{theorem}
  If $S_0(\knotexterior) \setminus S_0(\knotexterior)^{\hiota} \not =
  \emptyset$, then $\branchedcoverZHS{2}$ is not a lens space.
  In particular, $K$ is not a two--bridge knot.
\end{theorem}

\begin{remark}
  The similar statement also holds for a knot in an integral homology
  $3$--sphere.
\end{remark}

Here we give explicit representatives in conjugacy classes of
metabelian representations by using R. Riley's construction~\cite{Riley} and the
correspondence $\PhiII$.  
We choose the following presentation of a two--bridge knot group as
  \[
  \langle x, y\,|\, wx=yw \rangle,
  \]
where the meridians $x$ and $y$ are depicted in
Figure~\ref{fig:two_bridge} and $w$ is a word in $x$ and $y$. 
We will show the following theorem.

\begin{theorem}\label{thm:explicit_conj_class}
  Let $K$ be a two--bridge knot and $p$ the determinant of $K$, i.e.,
  $p=|\Delta_K(-1)|$. Then the slice
  $S_0(\knotexterior)$ is identified with the following set of
  conjugacy classes of $\SL$-representations:
  \[
  \left\{ \left.  
    \left[ \rho_k \right] 
          \,\right|\, 
    k = 1, \ldots, (p-1)/2 \right\} 
  \cup 
  \left\{ 
    \left[ \varphi_{\sqrt{-1}} \right]
  \right\},
  \]
  where the representation $\rho_k$ is given by
  \[
  x\mapsto \left(
    \begin{array}{cc}
      \sqrt{-1} & -\sqrt{-1} \\
      0  & -\sqrt{-1}
    \end{array}
  \right), \quad 
  y \mapsto \left(
    \begin{array}{cc}
      \sqrt{-1} & 0 \\
      -\sqrt{-1}(e^{k\pi\sqrt{-1}/p}-e^{-k\pi\sqrt{-1}/p})^2 & -\sqrt{-1}
    \end{array}
  \right)
  \]
  and $\varphi_{\sqrt{-1}}$ is the abelian representation given by
  \[
  \varphi_{\sqrt{-1}}(\mu)= \left(
    \begin{array}{cc}
      \sqrt{-1} & 0 \\
      0 & -\sqrt{-1}
    \end{array}
  \right).
\]
\end{theorem}

Riley had shown a construction of non-abelian $\SL$-representations of
two--bridge knot groups in~\cite{Riley} by using a polynomial equation
$\phi_K(t, u)=0$.  His construction gives every representative in each
conjugacy class of a non-abelian representation.  Note that the
irreducible metabelian representations are given by using roots of
$\phi_K(-1, u)=0$.  We also obtain the next statement about
$\phi_K(-1, u)$.
\begin{theorem}\label{thm:explicit_Riley_poly}
  We keep notations in Theorem \ref{thm:explicit_conj_class}.  The
  polynomial $\phi_K(-1, u)$ has distinct $(p-1)/2$ roots
  $\{(e^{k\pi\sqrt{-1}/p}-e^{-k\pi\sqrt{-1}/p})^2\,|\, k=1, \ldots,
  (p-1)/2\}$. Namely, $\phi_K(-1, u)$ is expressed as
  \[
  \phi_K(-1, u) = (-1)^{(p-1)/2} \prod_{k=1}^{(p-1)/2}\left\{u -
    (e^{k\pi\sqrt{-1}/p}-e^{-k\pi\sqrt{-1}/p})^2\right\}.
  \]
\end{theorem}
Note that a generator of $\pi_1(\branchedcoverZHS{2})$ in $S^3$ is
illustrated as in Figure~\ref{fig:two_bridge} (for details about the
generator, see~\cite[Chapter 12]{burde03:_knots}).
\begin{figure}[!ht]
  \begin{center}
    \includegraphics[scale=.7]{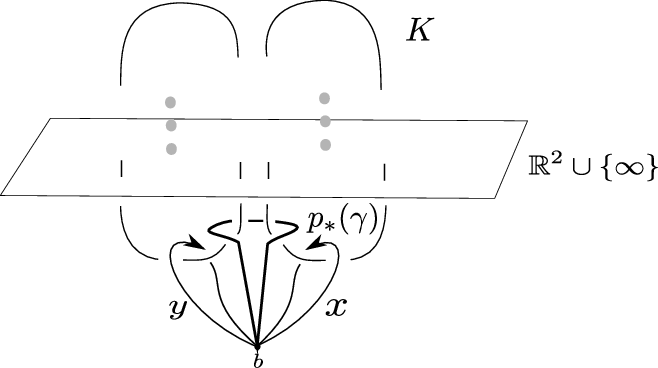}
  \end{center}
  \caption{The image of a generator $\gamma$ of
    $\pi_1(\branchedcoverZHS{2})$}\label{fig:two_bridge}
\end{figure}

\begin{proof}[Proof of Theorem \ref{thm:explicit_conj_class}]
  We focus on the non-abelian part in $S_0(\knotexterior)$.  By
  Theorem~$1$ in \cite{Riley}, every conjugacy class of non-abelian
  $\SL$-representations has a representative such that :
  \[
  \rho(x) = \left(
    \begin{array}{cc}
      \sqrt{t} & 1/\sqrt{t} \\
      0 & 1/\sqrt{t}
    \end{array}
  \right), \quad 
  \rho(y) = \left(
    \begin{array}{cc}
      \sqrt{t} & 0 \\
      -u\sqrt{t} & 1/\sqrt{t}
    \end{array}
  \right),
  \]
  where $t$ and $u$ satisfy the equation that $\phi_K(t, u) = 0$.  We
  now consider the case that $t=-1$ to describe elements in
  $S_0(\knotexterior)$.  
  We can take a generator $\gamma$ of the cyclic group $\pi_1(\branchedcoverZHS{2})$ 
  so that $p_*(\gamma) = xy^{-1}$ holds.
  The image of the homology class $[\gamma]$ by $p_*$ is null--homologous. 
  Then $\PhiII(\rho)(\gamma)$ is expressed as follows:
  \begin{align*}
    \PhiII(\rho)(\gamma)
    &= \left(\sqrt{-1}\right)^{p_*[\gamma]} \cdot \rho(p_*(\gamma)) \\
    &= \rho(xy^{-1})\\
    &= \left(
      \begin{array}{cc}
        u+1 & 1 \\
        u & 1
      \end{array}
    \right).
  \end{align*}
  By Theorem~\ref{thm:geom_S0}, the representation $\PhiII(\rho)$ is
  abelian.  Every abelian representation of
  $\pi_1(\branchedcoverZHS{2})$ is determined by the eigenvalues for
  the generator $\gamma$.  Since $\pi_1(\branchedcoverZHS{2})$ is the
  cyclic group with order $p$, the matrix
  $\PhiII(\rho)(\gamma)$ is conjugate to
  \[
  \left(
    \begin{array}{cc}
      e^{2\pi\sqrt{-1}/p} & 0 \\
      0 & e^{-2\pi\sqrt{-1}/p}
    \end{array}
  \right).
  \]
  Comparing the traces of $\PhiII(\rho)(\gamma)$ and the above
  diagonal matrix, we have that
  \[
    u = (e^{\pi\sqrt{-1}/p} - e^{-\pi\sqrt{-1}/p})^2.
  \] 
  Since $\PhiII$ gives a one--to--one correspondence on metabelian
  characters, we can conclude the statement. 
\end{proof}
\begin{proof}[Proof of Theorem~\ref{thm:explicit_Riley_poly}]
  If we fix a choice of square root of $-1$, the roots of $\phi_K(-1,
  u)$ correspond to the representatives of non-abelian part in
  $S_0(\knotexterior)$ by one--to--one.  From \cite{Riley}, the
  highest degree of $u$ in $\phi_K(t, u)$ is given by the determinant
  of $K$, i.e., $p=|\Delta_K(-1)|$.  Moreover it follows from the
  proof of \cite[Lemma 2]{Riley} that the coefficient of
  leading term on $u$ of $\phi_K(t, u)$ is equal to $(-1)^{(p-1)/2}$.  By the
  degree of $u$ and Theorem~\ref{thm:explicit_conj_class}, all roots
  of $\phi_K(-1, u)$ are given preciously by $(p-1)/2$ distinct real
  numbers:
  \[
  \{ (e^{k\pi\sqrt{-1}/p}-e^{-k\pi\sqrt{-1}/p})^2 
       \,|\, 
     k = 1, \ldots, (p-1)/2 \}.
  \]
  This completes the proof.  
\end{proof}

\medskip

\noindent {\bf Example}.  Let $K$ be the trefoil knot.  Then
$\phi_K(t, u)$ is given by $u-t-t^{-1}+1$.  
The root of $\phi_K(-1, u)$ is $u=-3$.  
The corresponding representation $\rho$ is expressed as
\[
\rho(x) = \left(
  \begin{array}{cc}
    \sqrt{-1} & -\sqrt{-1} \\
    0 & -\sqrt{-1}
  \end{array}
\right),\quad 
\rho(y)= \left(
  \begin{array}{cc}
    \sqrt{-1} & 0 \\
    3\sqrt{-1} & -\sqrt{-1}
  \end{array}
\right).
\]
Then the matrix $\PhiII(\rho)(xy^{-1})$ is given by
$\left(
  \begin{array}{cc}
    -2 & 1 \\
    -3 & 1
  \end{array}
\right).$
By direct calculations, we can see that
$\PhiII(\rho)(xy^{-1})^3 = \bm{1}$. Since $|\Delta_K(-1)|=3$, the formula in
Theorem~\ref{thm:explicit_Riley_poly} turns into
\[
u - (e^{\pi\sqrt{-1}/3}-e^{-\pi\sqrt{-1}/3})^2
= u + 3.
\]
This polynomial coincides with $\phi_K(-1, u)$.

Lemma~\ref{lemma:two_brige_S0} and Theorem~\ref{thm:geom_S0} show that 
the map $\hPhi$ is bijective for all two-bridge knots. 
Moreover the following holds on the surjectivity of $\hPhi$.
\begin{proposition}\label{prop:pretzel}
  If $K$ is a pretzel knot of type $(p, q, r)$, then the map $\hPhi$
  from $S_0(\knotexterior)$ to $X(\branchedcoverZHS{2})$ is
  surjective.
\end{proposition}
The key to Proposition $\ref{prop:pretzel}$ is that the 2--fold
branched cover $\branchedcoverZHS{2}$ along a pretzel knot of type
$(p, q, r)$ is the Brieskorn manifold of type $(p, q, r)$.  The
fundamental group of this Brieskorn manifold has a presentation with
four generators $s_1$, $s_2$, $s_3$ and $h$ (a central element) as
that of a Seifert manifold (for more details, 
see~\cite[Chapter 12]{burde03:_knots} and
\cite{nikolai02:_repres_spaces_of_seifer_fiber_homol_spher}).  
In fact, it is more convenient to work with another set of generators
$t_1 = s_1$, $t_2 = s_1s_2$ and $h$ when we consider the induced
action on the fundamental group by the covering transformation.  We
use this idea to prove Proposition~$\ref{prop:pretzel}$.
\begin{proof}
  The character variety $X(\branchedcoverZHS{2})$ is expressed as the
  union $X^{\mathrm{ab}}(\branchedcoverZHS{2}) \cup
  X^{\mathrm{irr}}(\branchedcoverZHS{2})$.  We will show that each
  part is contained in $\im \hPhi$.  By
  Proposition~\ref{prop:one_to_one},
  $X^{\mathrm{ab}}(\branchedcoverZHS{2})$ is contained in $\im
  \hPhi$.  Hence it suffices to show that
  $X^{\mathrm{irr}}(\branchedcoverZHS{2})$ is contained in $\im
  \hPhi$. Since $\im \hPhi$ is $X(\branchedcoverZHS{2})^{\tau}$,
  we check the inclusion $X^{\mathrm{irr}}(\branchedcoverZHS{2})
  \subset X(\branchedcoverZHS{2})^{\tau}$.  For 
  $\chi_\rho \in X^{\mathrm{irr}}(\branchedcoverZHS{2})$, 
  we have the following equivalence relations:
  \begin{align*}
    \tau^*(\chi_\rho) = \chi_\rho &\Leftrightarrow
    \chi_{\tau^* \rho} = \chi_\rho \\
    &\Leftrightarrow
    \rho \stackrel{conj}{\sim} \tau^* \rho \\
    &\Leftrightarrow \rho \in R^{\tau}(\branchedcoverZHS{2}).
  \end{align*}
  So, to complete the proof, it is enough to show that every
  irreducible representation is $\tau$-equivariant. The fundamental
  group $\pi_1(\branchedcoverZHS{2})$ has the following presentation:
  \[
  \langle s_1, s_2, s_3, h \,|\, s_1^p h=1, s_2^q h=1, s_3^r h=1,
  [s_i, h]=1\,(1\leq i \leq 3), s_1 s_2 s_3=1 \rangle.
  \]
  The action of $\tau$ on the generators $s_1, s_2, s_3, h$ is
  expressed as follows:
  \[
  \tau \co h \mapsto h,\, s_1 \mapsto s_1^{-1},\, s_2 \mapsto s_1
  s_2^{-1} s_1^{-1},\, s_3 \mapsto s_1 s_2 s_3^{-1} s_2^{-1} s_1^{-1}.
  \]
  We set $t_1$ and $t_2$ as $t_1 = s_1$ and $t_2 = s_1 s_2$.  Then
  three elements $h$, $t_1$ and $t_2$ generate the group
  $\pi_1(\branchedcoverZHS{2})$.  For the new generators, the action
  of $\tau$ is expressed as
  \[
  \tau \co h \mapsto h,\, t_1 \mapsto t_1^{-1},\, t_2 \mapsto t_2^{-1}.
  \]

  Since $\rho$ is irreducible, the image $\im \rho$ is a non--abelian
  subgroup in $\SL$.  Hence the central element $\rho(h)$ in $\im
  \rho$ is $\pm \I$.  By the relation $t_1^p h = 1$, the matrix
  $\rho(t_1)$ has a finite order, in particular it is hyperbolic. By
  taking conjugation, we can suppose that the irreducible
  representation $\rho$ is given by
  \[
  \rho(h) = \pm \I,\quad \rho(t_1) = \left(
    \begin{array}{cc}
      a & 0 \\
      0 & a^{-1}
    \end{array}
  \right),\quad \rho(t_2) = \left(
    \begin{array}{cc}
      s & t \\
      u & v
    \end{array}
  \right),
  \]
  where $ut \not = 0$.  We set a complex number $\delta$ as $\delta^2
  = \sqrt{-u/t}$.  Moreover, taking a conjugate by the matrix
  $\left(
    \begin{array}{cc}
      \delta & 0 \\
      0 & \delta^{-1}
    \end{array}
  \right)$, 
  we can assume that
  $\rho(t_2)$ is given by
  $
  \left(
    \begin{array}{cc}
      s & t \\
      -t & v
    \end{array}
  \right)
  $
  at the beginning.

  For a unit quaternion $\bm{k} = \left(
    \begin{smallmatrix}
      0 & \sqrt{-1} \\
      \sqrt{-1} & 0
    \end{smallmatrix}
  \right)$, we have 
  \[
  \bm{k}\rho(h)\bm{k}^{-1} = \rho(h),\ \bm{k}\rho(t_1)\bm{k}^{-1} =
  \rho(t_1)^{-1},\ \bm{k}\rho(t_2)\bm{k}^{-1} = \rho(t_2)^{-1}.
  \]
  Therefore it follows that $\tau^* \rho$ coincides with $\bm{k}
  \rho \bm{k}^{-1}$, i.e., $\rho$ is $\tau$--equivariant.
  This completes the proof.   
\end{proof}

\section*{Acknowledgment}
  The authors would like to thank Professor Xiao-Song Lin for giving
  substantial comments which drew their attention to metabelian
  representations when the first author stayed at the University of
  California, Riverside in 2006.  
  The second author would like to thank Professor
  Toshitake Kohno for his advice and support. .

  The authors also gratefully acknowledge the many technical and
  stylistic suggestions made by the referee. They helped us to improve
  all facets of the text.

%%%%%%%%%%%%%%%%%%%%%%%%%%%%%%%%%%%%%%%%%%%%%%%%%%%%%%%%%%%%%%%%%%%% 
% reference
%%%%%%%%%%%%%%%%%%%%%%%%%%%%%%%%%%%%%%%%%%%%%%%%%%%%%%%%%%%%%%%%%%%% 
\bibliographystyle{alpha}
\bibliography{geometry_S0} 
\end{document}